\documentclass[12pt,a4paper]{article}

\usepackage{amsmath}
\usepackage{amssymb}
\usepackage{mathtools}
\usepackage{hyperref}
\usepackage{mathrsfs} 

\usepackage[english]{babel}
\usepackage[T1]{fontenc}
\usepackage[utf8]{inputenc}

\usepackage{bm}
\usepackage{lmodern}
\usepackage[normalem]{ulem}

\usepackage[top=2.25cm, bottom=2.75cm, left=2.75cm, right=2.75cm]{geometry}
\usepackage{latexsym}
\usepackage{pdflscape}
\usepackage{fancyhdr}
\usepackage{setspace}

\usepackage{graphicx}
\usepackage[dvipsnames]{xcolor}
\usepackage{epstopdf}
\usepackage{epsf}
\usepackage{eepic}
\usepackage{tikz}
\usetikzlibrary{shadings,shapes,arrows,calc,positioning,shadings,decorations.markings, patterns}

\usepackage[numbers]{natbib}

\usepackage{amsmath,amssymb,amsfonts,amscd}
\usepackage{mathtools}

\usepackage{enumitem}
\usepackage{caption}
\usepackage[format=hang,margin=10pt]{subcaption}
\usepackage{listings}
\usepackage{varwidth}

\usepackage{algorithm}
\usepackage{algpseudocode}
\usepackage[numbered,framed]{matlab-prettifier}
\usepackage{amsthm, mathabx}
\usepackage[linguistics]{forest}
\forestset{
nice empty nodes/.style={
    for tree={l sep*=2},
    delay={where content={}{shape=coordinate}{}}
},
}

\usepackage{changepage}
\usepackage{xifthen}
\usepackage{todonotes}

\usepackage{hyperref}
\usepackage{cleveref}

\numberwithin{equation}{section}
\setlist[itemize,1]{label=$\bullet$}
\setlist[itemize,2]{label=$\triangleleft$}
\setlist[enumerate,1]{label=(\roman*)}
\setlist[enumerate,2]{label=(\arabic*)}

\definecolor{TUIl-orange}{RGB}{255, 121, 0}
\definecolor{TUIl-titleblue}{RGB}{0, 68, 121}
\definecolor{TUIl-textblue}{RGB}{0, 51, 88}
\definecolor{TUIl-green}{RGB}{0, 116, 122}
\definecolor{TUIl-grey}{RGB}{165, 165, 165}

\hypersetup{
    breaklinks=true,
    unicode=true,
    pdfpagelayout=OneColumn,
    bookmarksnumbered=true,
    bookmarksopen=true,
    bookmarksopenlevel=0,
    pdfborder={0 0 0},  
    colorlinks=true,
    linkcolor=Cerulean,    
    citecolor=Cerulean,
    urlcolor=Cerulean,
}

\algnewcommand\algorithmicinput{\textbf{Input:}}
\algnewcommand\AlgInput{\item[\algorithmicinput]}
\algnewcommand\algorithmicoutput{\textbf{Output:}}
\algnewcommand\AlgOutput{\item[\algorithmicoutput]}

\lstMakeShortInline"
\newtheoremstyle{dotless}{}{}{\itshape}{}{\bfseries}{}{ }{}
\newtheoremstyle{no-italic}{}{}{}{}{\bfseries}{}{ }{}
\theoremstyle{dotless}
\newtheorem{Theorem}{Theorem}[section]
\newtheorem{Example}[Theorem]{Example}
\newtheorem{Lemma}[Theorem]{Lemma}
\newtheorem{Definition}[Theorem]{Definition}
\newtheorem{Assumption}{Assumption}
\newtheorem{Remark}[Theorem]{Remark}
\newtheorem{Proposition}[Theorem]{Proposition}
\newtheorem{Corollary}[Theorem]{Corollary}
\newtheorem{Test Instance}[Theorem]{Test Instance}
\theoremstyle{no-italic}

\newcommand*{\R}{\mathbb{R}}

\newcommand*{\N}{\mathbb{N}}

\makeatletter
\def\smallunderbrace#1{\mathop{\vtop{\m@th\ialign{##\crcr
   $\hfil\displaystyle{#1}\hfil$\crcr
   \noalign{\kern3\p@\nointerlineskip}%
   {\scriptsize\upbracefill}\crcr\noalign{\kern3\p@}}}}\limits}
\makeatother

\DeclareMathOperator{\cl}{cl}			
\DeclareMathOperator{\argmina}{argmin^{I}}	
\DeclareMathOperator{\argminb}{argmin^{II}} 
\DeclareMathOperator{\argmin}{argmin}
\DeclareMathOperator{\wargmin}{wargmin}	

\usepackage{tikz}

\def\R{{\mathbb R}}
\def\N{{\mathbb N}}
\def\B{{\mathbb B}}
\def\dom{\textup{dom }}
\def\WMin{\textup{WMin}}

\def\Min{\textup{Min}}

\def\Int{\textup{int }}
\def\bd{\textup{bd }}
\def\cl{\textup{cl }}
\def\conv{\textup{conv }}
\def\gph{\textup{gph }}

\title{A Vectorization Scheme for\\ Nonconvex Set Optimization Problems}
\author{Gabriele Eichfelder\thanks{Institute of Mathematics, Technische Universität Ilmenau, Po 10 05 65, D-98684 Ilmenau, Germany,
{\texttt{\{gabriele.eichfelder,ernest.quintana-aparicio,stefan.rocktaeschel\}@tu-ilmenau.de}}}, Ernest Quintana\footnotemark[1] , Stefan Rocktäschel\footnotemark[1] }
\date{}

\begin{document}

\maketitle
\begin{abstract}
In this paper, we study a solution approach for  set optimization problems with respect to the lower set less relation. This approach can serve as a base for numerically solving set optimization problems by using established solvers from multiobjective optimization. Our strategy consists of deriving a parametric family of multiobjective optimization problems whose optimal solution sets  approximate, in a specific sense, that of the set-valued problem with arbitrary accuracy. We also examine particular classes of set-valued mappings for which the corresponding set optimization problem is  equivalent to a multiobjective optimization problem in the generated family. Surprisingly, this includes set-valued mappings with a convex graph. 
\end{abstract}


\vspace{2ex} \noindent {\small\textbf{Mathematics subject
classifications (MSC 2010):}}
90C48,  
90C29, 
90C59 


\section{Introduction}
\label{section:intro}

In set optimization one considers the minimization of set-valued  mappings  with  a  partially ordered image space. Problems of this type have received a lot of attention during the past few decades due to its wide number of both theoretical and practical applications. For example, from a theoretical point of view, it generalizes multiobjective optimization problems \cite{Jahn2011} and naturally arises in bilevel optimization \cite{Pil2016} and optimization under uncertainty \cite{idekobis2014}. On the other hand, it also has a direct impact in finance \cite{FeinsteinZacharyRudloff2015} and socio-economics \cite{Bao9,NeukelNorman2013}. We refer the reader to \cite{KTZ} for an in-depth overview of this research area.

\paragraph{Literature Review.}
In this paper, we are interested in the computational solution of set optimization problems for which the solution concepts are based on the so called set approach. A common feature of these solution concepts is that they are based on predefined order relations in the class of all nonempty subsets of the image space of the set-valued objective mapping. Here it is worth mentioning the so called lower set less  and upper set less relations.  Given its generality, the design of numerical methods for solving this type of problems is a challenging research topic. In the literature, different algorithmic approaches have already been studied. For unconstrained problems, derivative free strategies have been considered in \cite{jahn2015desc, jahn2018tree, KK2016}. The idea of these methods is to find a descent direction of the set-valued objective mapping among the points in a precomputed subset of the unit sphere, and then to  use a suitable line search procedure to improve the solution candidate.
These methods have in common that they deliver only one optimal solution, while in set optimization, similar to multiobjective optimization, the optimal solution set---and even the cardinality of optimal values---is in general infinite. 
The case in which the set-valued problem has a finite feasible set has been analyzed in \cite{guntherkobispopovici2019, guntherkobispopovici2019part2,  kobiskuroiwatammer2017, kobistam2018}. The main feature of these procedures is the use of clever comparison strategies of the images of the set-valued mapping that, in practice, avoids the comparison of every possible pair of sets.  Other solution methods considered exploit the particular structure of the set-valued objective mapping. In particular, for those that arise in multiobjective optimization under uncertainty,  scalarization methods have been studied in \cite{EhrgottIdeSchobel2014,eichfeldernieblingrocktaschel2019,  idekobis2014, IdeKobisKuroiwa2014, jiangcaoxiong2019, schmidtschobelthom2019}, a first order descent method was considered in \cite{bouzaquintanatammer2020}, and a branch and bound scheme has been examined in \cite{eichfeldernieblingrocktaschel2019}. 

In contrast to the approaches previously mentioned, in this paper we are particularly concerned with vectorization strategies. By this, we mean any procedure that makes use of some kind of relaxation of the original set optimization problem that replaces it with a finite-dimensional multiobjective optimization problem for which numerical solution methods are already available.
Research in this direction is, to the best of our knowledge, more limited in the literature. In fact, we are only aware of two references in this case, see \cite{eichfelderrocktaeschel2021, lohneschrage2013}. The algorithm in \cite{lohneschrage2013} is designed for problems on which the graph of the set-valued objective mapping is polyhedral and the solution concept is the one associated to the lower set less relation.   The more recent work \cite{eichfelderrocktaeschel2021} assumes that the images of the set-valued mapping are convex and compact, and  deals with both the upper and lower set less relations. The idea there is to consider a discretized version of a vectorization result derived in \cite{Jahn2015A}, which shows the equivalence (in a specific sense) of set optimization problems and a corresponding infinite dimensional vector optimization problem. This results in a parametric family of multiobjective problems whose solution sets approximate that of the set optimization problem, and the quality of this approximation is controlled by the fineness of the discretization.                                                                                        

\paragraph{Contributions.}
The main contribution of this paper is a vectorization scheme for solving general nonconvex set optimization problems with the lower set less relation. Inspired by the solution concepts based on the so called vector approach  we construct, for any natural number, a corresponding multiobjective optimization problem. Thus, by doing this, we obtain a parametric family of multiobjective optimization  problems with the natural numbers as the set of parameters. Thereby, the vector approach is obtained as a particular case. The results derived for the proposed scheme are of the following two types:
\begin{itemize}
    \item \emph{Approximation properties.} We show that the optimal  solution sets of our multiobjective subproblems give inner approximations to that of the set optimization problem, and that these approximations are monotone increasing with respect to the parameter. Furthermore, we prove that the quality of the solutions obtained by these inner approximations are good enough in a specific sense. At the same time, outer approximations are provided by means of the approximate solutions of the multiobjective subproblems. As a consequence, we are able to establish that the optimal  solution set of the set optimization problem can be approximated with arbitrary accuracy by those of the multiobjective optimization  problems in our parametric family. 
    
    \item \emph{Finiteness of the vectorization scheme.} We study different classes of set-valued mappings for which the corresponding set optimization problem is equivalent to one of our multiobjective optimization  subproblems. In this case, the inner approximation determined by the multiobjective optimization problem is already the full set of optimal solutions of interest. In a specific topological sense, we are able to show that the class of set-valued mappings satisfying this property is not small, and that it includes different types usually considered in the literature.  In particular, we show that this is the case if the set-valued mapping has polytopal values, or values with a finite number of minimal elements, or a convex graph.
\end{itemize}

Given the guaranteed approximation properties mentioned above, we can replace the set optimization problem with one of our multiobjective optimization subproblems for a specific parameter. The subproblems have a linear objective function and inclusion constraints with respect to the graph of the set-valued objective mapping. Thus, our approach is computationally friendly in the sense that no extra complexity is added to the multiobjective subproblems, other than those that already come from the description of the set-valued objective mapping.

A clear advantage of our method over those in the literature is that we do not require the convexity of the images of the set-valued mapping (as in \cite{eichfelderrocktaeschel2021, lohneschrage2013}), and we avoid making expensive computational comparisons between sets (as required in \cite{jahn2015desc, jahn2018tree, KK2016, guntherkobispopovici2019, guntherkobispopovici2019part2,  kobiskuroiwatammer2017, kobistam2018}). 

\paragraph{Organization.}
The paper is organized as follows: Section \ref{section:prelims} introduces the main concepts and results that set the stage of the work. In Section \ref{section:vectorization}, we introduce our vectorization scheme and study the main connections to set optimization problems. These connections then motivate the study of classes of set-valued mappings for which our vectorization is exact in a specific sense. This is precisely the content of Section \ref{section:fdvp}. We conclude in Section \ref{section:Conclusions} with some final remarks.

\section{Preliminaries}
\label{section:prelims}

We start this section by clarifying the notation of the paper. First, for a set $A \subseteq \R^m$, we denote its interior, closure, boundary, convex hull and cardinality by $\Int A$, $\cl A$, $\bd A$, $\conv A,$ and $|A|,$ respectively.  As usual, the elements in $\R^m$ are considered as column vectors, and we denote the transpose operator with the symbol $\top.$  However, we sometimes abuse this notation when we consider vectors that are formed by other vectors together. Thus, for example, given  $x \in \R^n$ and $y \in \R^m,$ we may write $(x,y)$ instead of $(x^\top,y^\top)^\top.$ We also denote by $0_{m\times n}$ the zero matrix with $m$ rows and $n$ columns, and by $\mathcal{I}_m$ the identity matrix of dimension $m.$  Furthermore, we denote the standard Euclidean norm in $\R^m$ by $\|\cdot\|$ and its corresponding unit ball by $\mathbb{B}.$ Moreover, for a set-valued mapping $F:\R^n \rightrightarrows \R^m$ and $A \subseteq \R^n,$ we denote by $F(A)$ the set $\bigcup_{x \in A} F(x).$  Finally, for $p\in \N,$ we set $[p]:=\{1,\ldots,p\}.$ 

%

Recall that $K \subseteq \R^m$ is a cone if  for all
$t\geq 0$ and $y \in K$ it holds $t y\in K,$ and that a cone $K$ is convex if $K + K = K,$ pointed if $K\cap (-K)=\{0\},$ and solid if $\Int K \neq \emptyset.$ For a cone $K,$ we denote with $K^*$ the dual cone defined as
$$
K^*:=\{v \in \R^m \mid \forall\; y\in K: v^\top y\geq 0\}.
$$
For a solid cone 
$K \subseteq \R^m$  and $y^1,y^2 \in \R^m$
we define  binary relations by
\[\begin{array}{rcl}
y^1\preceq_K y^2: &\Longleftrightarrow &y^2-y^1\in K, \\
y^1\precneq_K y^2: &\Longleftrightarrow &y^2-y^1\in K\setminus\{0\},\\
y^1\prec_K y^2: &\Longleftrightarrow &y^2-y^1\in \Int K\\
\end{array}
\]
We  also write $y^1 \npreceq_K y^2$ and $y^1 \nprec_K y^2$ if the inequalities $y^1\preceq_K y^2$ and $y^1\prec_K y^2$ are not satisfied, respectively.
%
%
%
%
%

If the cone $K$ in the previous definition is convex and pointed, it is well known that $\preceq_K$ defines a partial order in $\R^m,$ see, for example,  \cite{Jahn2011}. By using the partial ordering introduced by a convex and pointed cone $K$ we can define  minimal and weakly minimal elements of a nonempty 
set $A \subseteq \R^m$. For the weakly minimal elements we need to assume additionally that the cone $K$ is solid. 
The set of minimal elements of $A$ with respect to $K$ is then defined as
$$\Min (A,K):= \{y \in A\mid \left(y -  K\right)\cap A =\{y\}\},$$
and  the set of weakly minimal elements of $A$ with respect to $K$ is defined as
$$\WMin (A,K):= \{y \in A\mid \left(y- \Int K\right)\cap A =\emptyset\}.$$ 
\begin{Proposition}(\cite[Theorem $3.2.9$]{SawaragiNakayamaTanino1985})\label{prop:domination property}
Let $A \subseteq \R^m$ be nonempty and compact, and suppose that $K$ is closed, convex, and pointed. Then, $A$ satisfies the so called domination property (or external stability property) with respect to $K$, that is, $\Min (A,K)\neq \emptyset$ and $$A + K = \Min(A,K) + K.$$
\end{Proposition}

Let now $\Omega \subseteq \R^n$ be nonempty, and consider a vector-valued  function $f:\R^n \rightarrow \R^m.$ Furthermore, let $K\subseteq \R^m$ be a closed, convex, pointed, and solid cone. The multiobjective optimization problem associated to this data is then  
\begin{equation}\label{eq:vpg}\tag{$\mathcal{VP}$}
\begin{array}{rl}
\preceq_K\textrm{-}&\min\limits_{x} \; f(x) \\
&\;\textup{s.t.} \;\; x \in \Omega,\\
\end{array}  
\end{equation} 
and its (approximate) optimal solutions are understood in the sense of Definition \ref{def:esolutions vp} below. As we will see later in the paper, we can relate (approximate) optimal solutions of a specific multiobjective optimization problem to the optimal solution set of our set optimization problem under consideration.

\begin{Definition}[\cite{Kutateladze1979}]\label{def:esolutions vp}
Let $K\subseteq \R^m$ be a closed, convex, pointed, and solid cone, and suppose that an element $e \in \Int K$ has been fixed.  Furthermore, let $\varepsilon\geq 0,$ and consider a point $\bar{x} \in \Omega.$   
\begin{enumerate}
    \item We say that $\bar{x}$ is an $\varepsilon$-weakly minimal solution of $\eqref{eq:vpg}$ if 
        \begin{equation}
        \nexists \; x\in \Omega: f(x)\prec_K f(\bar{x}) - \varepsilon e.
        \end{equation}
        When $\varepsilon = 0,$ we just say that $\bar{x}$ is a weakly minimal solution of  \eqref{eq:vpg}. The set of $\varepsilon$-weakly minimal solutions is denoted by $\varepsilon\textrm{-}\wargmin \eqref{eq:vpg}$ for $\varepsilon >0,$ and by  $\wargmin \eqref{eq:vpg}$ for $\varepsilon = 0.$ 
    \item We say that $\bar{x}$ is an $\varepsilon$-minimal solution of $\eqref{eq:vpg}$ if
        \begin{equation}
        \nexists \; x\in \Omega: f(x) \precneq_K f(\bar{x}) - \varepsilon e.
        \end{equation}
        When $\varepsilon = 0,$ we just say that $\bar{x}$ is a minimal solution of  \eqref{eq:vpg}. The set of  $\varepsilon$-minimal solutions is denoted by $\varepsilon \textrm{-}\argmin \eqref{eq:vpg}$ for $\varepsilon >0,$ and by  $\argmin \eqref{eq:vpg}$ for $\varepsilon = 0.$ 
\end{enumerate}
\end{Definition}

The remaining of this section is devoted to the definition of the set optimization problem we consider in the paper. A main ingredient in this definition is the use of so called set relations in order to compare the images of the set-valued objective mapping when determining an optimal solution. Thus, we begin by establishing those that are used  in the paper. 
We refer the reader to \cite{jahnha2011} and the references therein for an overview of other existing set relations.

\begin{Definition}\label{def-set-relation}
Let $K \subseteq \R^m$ be a closed, convex, pointed, and solid cone. The following binary relations are defined on the power set of $\R^m:$
\begin{enumerate}
    \item $\forall\; A,B \subseteq \R^m:  A\preceq^\ell_K B: \Longleftrightarrow B\subseteq A+K,$
    \item $\forall\; A,B \subseteq \R^m:  A\precneq^\ell_K B: \Longleftrightarrow B\subseteq A+K\setminus \{0\},$
    \item $ \forall\; A,B \subseteq \R^m:  A\prec^\ell_K B: \Longleftrightarrow B\subseteq A+ \Int K.$ 
\end{enumerate}
For $A,B \subseteq \R^m,$ we write $A\npreceq^\ell_K B$ and $A\nprec^\ell_K B$ if the inequalities $A\preceq^\ell_K B$ and $A\prec^\ell_K B$ are not satisfied, respectively.
\end{Definition}

The binary relation $\preceq^\ell_K$ defined above is the so called lower set less relation associated to the cone $K$ and it has been widely used in the literature, see for instance \cite{KTZ} and the references therein. One can also observe that the relations $\precneq^\ell_K$ and $\prec^\ell_K$ constitute, respectively, strong and strict versions of it. Moreover, note that
$$A\preceq^\ell_K B \Longleftrightarrow \forall \; b \in B,\; \exists\; a \in A: a \preceq_K b.$$ Thus, intuitively, comparing sets with respect to this set relation corresponds to comparing their best elements with respect to the partial order $\preceq_K.$ It is well known that $\preceq^\ell_K$ is a preorder, but in general it fails to be antisymmetric, see \cite[Proposition 3.1]{jahnha2011}. Likewise, $\precneq^\ell_K$ and $\prec^\ell_K$ are transitive  but not necessarily reflexive or antisymmetric.

Next, we recall some standard concepts from set-valued and convex analysis that will be important for our main results. For a set-valued mapping $F:\R^n \rightrightarrows \R^m$, the domain and graph of $F$ are defined, respectively, by 
$$\dom F: = \{x\in \R^n \mid F(x)\neq \emptyset\} \textrm{ and } \gph F: = \{(x,y) \in \R^n\times \R^m \mid y \in F(x)\}.$$
Moreover, we say that $F$ is locally bounded at $\bar{x}$ if there exists $\alpha>0$ and a neighborhood $U$ of $\bar{x}$ such that
$F\left(U\right) \subseteq \alpha \mathbb{B}$. 
%

Unless otherwise stated, throughout the rest of this paper we work with the following setup:
\begin{Assumption}
Let $\Omega \subseteq \R^n$ be nonempty and closed, and $F: \R^n \rightrightarrows \R^m$ be a given set-valued mapping such that $\Omega\subseteq  \dom F$ and $F(x)$ is compact for every $x \in \Omega.$ Furthermore, let  $K \subseteq \R^m$ be a closed, convex, pointed, and solid cone, and $e \in   \Int K$ be a given element.
\end{Assumption} 
For a discussion on solidness of cones and possible replacement concepts we refer to the recent paper \cite{KhazayelFarajzadehGuenther2021}. 
Now, under these assumptions, we consider the set optimization problem:
\begin{equation}\label{sp}\tag{$\mathcal{SP}$}
\begin{array}{rl}
\preceq^\ell_K\textrm{-}&\min\limits_{x} \; F(x) \\
                        &\;\textup{s.t.} \;\; x \in \Omega.\\
\end{array}  
\end{equation} 
The solutions of \eqref{sp} are understood in the following sense:
\begin{Definition}
Let $\varepsilon \geq 0$ and consider a point $\bar{x} \in \Omega.$  
\begin{enumerate}
   \item We say that $\bar{x}$ is an $\varepsilon$-weakly minimal solution of $\eqref{sp}$ if 
        \begin{equation*}
        \nexists \; x\in \Omega: F(x)\prec^\ell_K F(\bar{x}) - \varepsilon e.
        \end{equation*}
        When $\varepsilon = 0,$ we just say that $\bar{x}$ is a weakly minimal solution of  \eqref{sp}. The set of $\varepsilon$-weakly minimal solutions is denoted by $\varepsilon\textrm{-}\wargmin \eqref{sp}$ for $\varepsilon >0,$ and by  $\wargmin \eqref{sp}$ for $\varepsilon = 0.$ 
    \item We say that $\bar{x}$ is a type-one $\varepsilon$-minimal solution of $\eqref{sp}$ if
        \begin{equation*}
        \forall \; x\in \Omega: F(x) \preceq^\ell_K F(\bar{x}) - \varepsilon e \Longrightarrow F(\bar{x}) - \varepsilon e \preceq^\ell_K F(x).
        \end{equation*}
        When $\varepsilon = 0,$ we just say that $\bar{x}$ is a type-one minimal solution of  \eqref{sp}. The set of type-one $\varepsilon$-minimal solutions is denoted by $\varepsilon\textrm{-}\argmina \eqref{sp}$ for $\varepsilon >0,$ and by  $\argmina \eqref{sp}$ for $\varepsilon = 0.$ 
    \item We say that $\bar{x}$ is a type-two $\varepsilon$-minimal solution of $\eqref{sp}$ if
        \begin{equation*}
        \nexists \; x\in \Omega: F(x) \precneq^\ell_K F(\bar{x}) - \varepsilon e.
        \end{equation*}
        When $\varepsilon = 0,$ we just say that $\bar{x}$ is a type-two minimal solution of  \eqref{sp}. The set of type-two $\varepsilon$-minimal solutions is denoted by $\varepsilon\textrm{-}\argminb \eqref{sp}$ for $\varepsilon >0,$ and by  $\argminb \eqref{sp}$ for $\varepsilon = 0.$ 
\end{enumerate}
\end{Definition}
In particular, we are interested in computing (or approximating) the sets $\wargmin \eqref{sp},$ $\argmina \eqref{sp},$ and $\argminb \eqref{sp}.$  The concepts involving $\varepsilon$ are only of approximate type, but also constitute an important tool in our analysis. It is worth mentioning that, concerning minimal solutions of set optimization problems, the traditional concept in the literature is that of type-one defined above. Our motivation for considering type-two minimal solutions comes from its applications in  multiobjective optimization  under uncertainty, where such a solution concept is introduced  for set-valued mappings with a particular structure, see \cite{idekobis2014, IdeKobisKuroiwa2014}. Moreover, from the beginning we will show that type-two minimal solutions have better approximation properties than those of type-one, see Proposition \ref{prop:ssvp subset sssp} and Example \ref{ex:t-one behave bad}. Therefore, in this paper we focus on the weakly minimal and type-two minimal solutions of \eqref{sp}.

In the following proposition we clarify the relationships between the different solution concepts.
\begin{Proposition}
Let $\varepsilon \geq 0$. It holds:
$$\varepsilon\textrm{-}\argmina \eqref{sp} \subseteq \varepsilon\textrm{-}\argminb \eqref{sp} \subseteq \varepsilon\textrm{-}\wargmin \eqref{sp}. $$
\end{Proposition}
\begin{proof}
In order to see the first inclusion, suppose that for some point $\bar{x}\in \varepsilon\textrm{-}\argmina \eqref{sp}$ we have $\bar{x}\notin \varepsilon\textrm{-}\argminb \eqref{sp}$. Then, there exists $x\in \Omega$ such that $$F(\bar{x}) - \varepsilon e \subseteq F(x) + K\setminus\{0\} \subseteq  F(x) + K.$$ Because $\bar{x} \in \argmina \eqref{sp},$ we deduce that $F(x) \subseteq F(\bar{x}) - \varepsilon e + K .$ Therefore, $$F(\bar{x}) - \varepsilon e \subseteq F(x) + K\setminus\{0\} \subseteq F(\bar{x}) - \varepsilon e + K + K\setminus\{0\}.$$ Thus, in particular, as $K$ is pointed, we find that $F(\bar{x}) \subseteq F(\bar{x})+ K\setminus\{0\}$.  According to Proposition \ref{prop:domination property}, we now get the existence of $\bar{y}\in \Min (F(\bar{x}),K)$. However, since $F(\bar{x}) \subseteq F(\bar{x})+ K\setminus\{0\}$, it is possible to find an element $y\in F(\bar{x})$ with $y\precneq_K \bar{y}.$ This contradicts the minimality of $\bar{y}$.

The second inclusion is an immediate consequence of the definitions and the fact that $\Int K \subseteq K\setminus\{0\}.$
\end{proof}

We conclude the section with a result that establishes relationships between the set of minimal  
solutions of \eqref{sp} and their approximate counterparts. 

\begin{Proposition} \label{prop:lalita=wintersectionepsilon}
It holds:
\begin{enumerate}
    \item \label{item:wargmin=intersectionepsilon} $\wargmin \eqref{sp} = \bigcap\limits_{\varepsilon >0 } \varepsilon\textrm{-}\wargmin \eqref{sp}.$
    
    \item \label{item:argmin=intersectionepsilon} $\argminb \eqref{sp} \subseteq \bigcap\limits_{\varepsilon >0 } \varepsilon\textrm{-}\argminb \eqref{sp}.$
\end{enumerate}
\end{Proposition}
\begin{proof}
For statement \ref{item:wargmin=intersectionepsilon}, see \cite[Theorem 4.2]{DhingraLalitha2017}. Statement \ref{item:argmin=intersectionepsilon} is immediate from the definition.
\end{proof}

\section{The Vectorization Scheme}
\label{section:vectorization}

This section is devoted to the formal description of our vectorization scheme and its main properties. With this in mind, we start by recalling the solution concepts for set optimization problems that are based on the so called vector approach  \cite[Section  2.6]{KTZ} as this delivers a motivation for our strategy. There, a pair $(\bar{x}, \bar{y})$ is said to be a (weak) minimizer of \eqref{sp} if and only if $(\bar{x},\bar{y})$ is a (weakly) minimal solution of the multiobjective optimization problem
\begin{equation}\label{vp1}
\begin{array}{rl}
\preceq_{K}\textrm{-} &\min\limits_{x,y} \; y \\
                      &\; \textup{s.t.} \; \;  y \in  F(x), \\
                      &  \qquad x \in \Omega.
\end{array} 
\end{equation} Note that, because we are mainly interested in optimal solutions in the decision space, the vector $y$ in \eqref{vp1} can be seen as a vector of slack variables. In the literature, connections between the solution concepts based on the vector and the set approach have already been studied, see \cite[Lemma 2.6 and Proposition 2.10]{HR2007}. There, it is established that
\begin{equation}\label{eq:minimizer is minimal}
    (\bar{x}, \bar{y}) \textrm{ is a weak minimizer of } \eqref{sp} \Longrightarrow \bar{x} \in \wargmin \eqref{sp}. 
\end{equation} 
This makes the process of finding solutions to \eqref{sp} by means of \eqref{vp1} an attractive computational idea. However, simple examples show that the converse implication in \eqref{eq:minimizer is minimal} does not hold and that, in general, the gap between the set of weakly minimal solutions and those that can be computed by solving \eqref{vp1} can be arbitrarily large. The main contribution of this paper is an attempt at remedying this limitation by introducing a generalization of problem \eqref{vp1} consisting of adding several slack vectors, instead of just one. Formally, for $p\in \N,$ we consider the multiobjective optimization problem
\begin{equation}\label{vpp}
\begin{array}{rl}
\preceq_{K^p}\textrm{-} & \min\limits_{x,y^1, \ldots,y^p} \; \begin{pmatrix}
y^1 \\ \vdots \\y^p
\end{pmatrix}   \tag{$\mathcal{VP}_p$} \\
  \\
                        & \quad \textup{s.t.} \;\; y^i  \in F(x) \textrm{ for each }  i \in [p], \\
                        & \qquad \quad x \in \Omega,
\end{array} 
\end{equation} where $K^p:= \prod_{i = 1}^p K \subseteq\R^{mp}.$ 

It is now easy to see that \eqref{vp1} is just \eqref{vpp} for $p = 1.$ This fact, together with statement \eqref{eq:minimizer is minimal}, clearly motivates a study on the relationships between \eqref{sp} and \eqref{vpp}. However, before proceeding, we make some additional clarifications. First, naturally, the approximate solutions of \eqref{vpp} are considered with respect to the vector  $e^p:= (e,  \ldots , e) \in \Int K^p.$ Moreover, for $\varepsilon > 0,$ we denote by $\varepsilon\textrm{-}\wargmin_x\eqref{vpp}$ the projection of the set $\varepsilon\textrm{-}\wargmin \eqref{vpp}$ onto $\R^n,$ that is,
$$\varepsilon\textrm{-}\wargmin_x\eqref{vpp}:= \{x' \in \R^n \mid \exists\; y^1,\ldots y^p \in F(x') : (x',y^1,\ldots, y^p) \in \varepsilon\textrm{-}\wargmin \eqref{vpp}\}.$$ Similarly,  we denote by $\varepsilon\textrm{-}\argmin_x\eqref{vpp}$ the projection of $\varepsilon\textrm{-}\argmin\eqref{vpp}$ onto $\R^n.$ Finally, for our analysis it will be helpful to introduce the vector-valued functional $f^p: \R^n \times \prod_{i = 1}^p \R^m \rightarrow \prod_{i = 1}^p \R^m $ and the set-valued mapping $F^p:\R^n \rightarrow \prod_{i = 1}^p \R^m$ defined respectively as 
$$\forall \;x \in \R^n ,y^1,\ldots,y^p \in \R^m:   f^p\left(x,y^1,\ldots,y^p\right) := \begin{pmatrix}
y^1 \\ \vdots \\y^p
\end{pmatrix}, \;\; F^p(x):= \prod\limits_{i = 1}^p F(x).$$ 

By using these maps we can rewrite  \eqref{vpp} as
$$
\begin{array}{rrl}
\preceq_{K^p}\textrm{-} &  \min\limits_{z} &f^p(z)\\ 
&\mbox{s.t.}&z\in \gph F^p.
\end{array}$$Our starting point is now a simple proposition that shows a monotonicity property of the solution sets of the  multiobjective optimization  problems \eqref{vpp} with respect to $p.$ 

\begin{Proposition}\label{prop:sset monot}
Let $p_1,p_2 \in \N$ and $\varepsilon \geq 0$ be given, and suppose that $p_1 \leq p_2.$ Then,
\begin{enumerate}
    \item \label{item:monot p} $\varepsilon\textrm{-} \wargmin_x \;(\mathcal{VP}_{p_1}) \subseteq \varepsilon\textrm{-} \wargmin_x \;(\mathcal{VP}_{p_2}),$
    \item \label{item:monot p for min} $ \varepsilon\textrm{-} \argmin_x \;(\mathcal{VP}_{p_1}) \subseteq \varepsilon\textrm{-} \argmin_x \;(\mathcal{VP}_{p_2}).$
\end{enumerate}
\end{Proposition}
\begin{proof}
For the proof of \ref{item:monot p}, let $\bar{x} \in \varepsilon\textrm{-} \wargmin_x \;(\mathcal{VP}_{p_1}).$ Then, we can find $\{\bar{y}^1,\ldots,\bar{y}^{p_1}\} \subseteq F(\bar{x})$ such that $$(\bar{x},\bar{y}^1,\ldots,\bar{y}^{p_1}) \in \varepsilon\textrm{-} \wargmin \;(\mathcal{VP}_{p_1}).$$ It is then easy to check that the vector where we add $p_2-p_1$ copies of $\bar y^1$, i.e., $\bar z:=
(\bar{x},\bar{y}^1,\ldots,\bar{y}^{p_1},\bar{y}^1,\ldots, \bar{y}^1)$ is an element of  $\varepsilon\textrm{-} \wargmin \;(\mathcal{VP}_{p_2}),$ from which we deduce that $\bar{x} \in \varepsilon\textrm{-} \wargmin_x \;(\mathcal{VP}_{p_2}).$

For the proof of \ref{item:monot p for min}, let $\bar{x} \in \varepsilon\textrm{-} \argmin_x \;(\mathcal{VP}_{p_1}).$ Then, we can find $\{\bar{y}^1,\ldots,\bar{y}^{p_1}\} \subseteq F(\bar{x})$ such that $(\bar{x},\bar{y}^1,\ldots,\bar{y}^{p_1}) \in \varepsilon\textrm{-} \argmin \;(\mathcal{VP}_{p_1}).$ 
Now let $\bar z\in \gph F^{p_2}$ be defined as in part \ref{item:monot p} of the proof and
assume that there exists $z=(x,y^1,\ldots,y^{p_1},\ldots,y^{p_2})$ with $z\in \gph F^{p_2}$ such that $ f^{p_2}(z) \precneq_{K^{p_2}} f^{p_2}(\bar{z})- \varepsilon e^{p_2}$.
This implies that $y^i \preceq_K \bar{y}^i - \varepsilon e$ for all $i\in [p_1]$. Furthermore, it is $y^{p_1+i}\precneq_K \bar{y}^1- \varepsilon e$ for an $i\in [p_2-p_1]$ or $y^j\precneq_K \bar{y}^j - \varepsilon e$ for a $j\in [p_1]$. In the fist case it is $f^{p_1}(x,y^{p_1+i},y^2,y^3,\ldots, y^{p_1}) \precneq_{K^{p_1}} f^{p_1}(\bar{x},\bar{y}^1,\ldots,\bar{y}^{p_1})-  \varepsilon e^{p_1}$ and in the second case it is $f^{p_1}(x,y^1,\ldots,y^{p_1}) \precneq_{K^{p_1}} f^{p_1}(\bar{x},\bar{y}^1,\ldots,\bar{y}^{p_1})-  \varepsilon e^{p_1}$. In both cases this contradicts $(\bar{x},\bar{y}^1,\ldots,\bar{y}^{p_1}) \in \varepsilon\textrm{-} \argmin \;(\mathcal{VP}_{p_1})$.
\end{proof}
The next result shows that a property similar to the one in \eqref{eq:minimizer is minimal} still holds for 
\eqref{vpp}.
\begin{Proposition}\label{prop:ssvp subset sssp}
Let $p \in \N$  and $\varepsilon \geq 0$ be given. Then, 
\begin{enumerate}
    \item \label{item:monot epsilon} $\varepsilon\textrm{-} \wargmin_x \eqref{vpp} \subseteq \varepsilon\textrm{-} \wargmin \eqref{sp},$
    \item \label{item:monot epsilon for min} $\varepsilon\textrm{-} \argmin_x \eqref{vpp} \subseteq \varepsilon\textrm{-} \argminb \eqref{sp}.$
\end{enumerate}
\end{Proposition}
\begin{proof}
For the proof of \ref{item:monot epsilon}, let us fix $p \in \N$ and consider an element $\bar{x} \in \varepsilon\textrm{-} \wargmin_x \eqref{vpp}.$ Again, there exists $\{\bar{y}^1,\ldots,\bar{y}^p\} \subseteq F(\bar{x})$ such that $ (\bar{x},\bar{y}^1,\ldots,\bar{y}^p ) \in \varepsilon\textrm{-} \wargmin \eqref{vpp}.$  Suppose now that $\bar{x} \notin \varepsilon\textrm{-}\wargmin \eqref{sp}.$ Then, we can find  $x \in \Omega$ for which the inequality $F(x) \prec^\ell_K F(\bar{x})- \varepsilon e$ is satisfied. Thus, in particular, we have that $$F(x) \prec^\ell_K\{\bar{y}^1,\ldots,\bar{y}^{p_1}\} - \varepsilon e.$$ This implies 
the existence of $y^1,\ldots,y^p \in F(x)$ such that 
$$\forall\; i \in [p]: y^i \prec_K \bar{y}^i - \varepsilon e,$$ from which we deduce that $f^p\left(x,y^1,\ldots,y^p\right) \prec_{K^p} f^p\left(\bar{x},\bar{y}^1,\ldots,\bar{y}^p\right) - \varepsilon e^p.$ This contradicts the supposed $\varepsilon$-weak minimality of $(\bar{x},\bar{y}^1,\ldots,\bar{y}^p)$ for \eqref{vpp}. The proof of \ref{item:monot epsilon for min} is analog.
\end{proof}

The following example shows that a statement like the one in Proposition \ref{prop:ssvp subset sssp} \ref{item:monot epsilon for min} does not hold for type-one minimal solutions.
\begin{Example}\label{ex:t-one behave bad}
Let $\Omega=[0.25,0.5]$, $Y=\R^2$, $K=\R^2_+$ and consider $F: \R \rightrightarrows \R^2$ given by $F(x):= 
      \conv \left\lbrace (1,0)^\top, (0,1)^\top, (x,x)^\top \right\rbrace   \textrm{ if } x\in \Omega,
$ and $F(x)=
      \emptyset$   if  $x \notin \Omega$.
Then, it is easy to check that $ \argmina \eqref{sp}=\{ 0.25 \} $ and $\argminb \eqref{sp}=\Omega$. 

However, for every $p\in\N$ it is $(0.5,  (1,0)^\top, \ldots,  (1,0)^\top)\in \argmin \eqref{vpp}$ (with $p$ copies of $(1,0)^\top$). Hence, we have $0.5\in \argmin_x \eqref{vpp}\setminus \argmina \eqref{sp}$ for all $p\in\N,$ and therefore  $\argmin_x \eqref{vpp} \nsubseteq \argmina \eqref{sp} .$
\end{Example}
\begin{Proposition}
Let $p \in \N$. Then, the following statements hold:
\begin{enumerate}
    \item For every $\bar{x} \in \wargmin_x \eqref{vpp}$ there exists $\bar y^1,\ldots,\bar y^p\in$
    $\Min(F(\bar{x}),K)$ such that 
    \begin{equation*}
     (\bar{x},\bar{y}^1,\ldots,\bar{y}^p ) \in \wargmin \eqref{vpp}.
    \end{equation*} Conversely, suppose that $(\bar{x},\bar{y}^1,\ldots,\bar{y}^p ) \in \wargmin \eqref{vpp}.$ Then, for some $i \in [p]$ we have $\bar{y}^i \in \WMin(F(\bar{x}),K).$
    \item For every $\bar{x} \in  \argmin_x \eqref{vpp}$ there exists $\bar y^1,\ldots,\bar y^p\in \Min(F(\bar{x}),K)$ such that 
    \begin{equation*}
     (\bar{x},\bar{y}^1,\ldots,\bar{y}^p ) \in  \argmin \eqref{vpp}.
    \end{equation*} Conversely, if $(\bar{x},\bar{y}^1,\ldots,\bar{y}^p ) \in \argmin \eqref{vpp},$ then $\bar y^1,\ldots,\bar y^p\in
    \Min(F(\bar{x}),K).$
\end{enumerate}
\end{Proposition}
 
\begin{proof}
 $(i)$ If $\bar{x} \in \wargmin_x \eqref{vpp},$ by the definition we can find $\{y^1,\ldots,y^p\} \subseteq F(\bar{x})$ such that  $(\bar{x},y^1,\ldots,y^p ) \in \wargmin \eqref{vpp}.$ Furthermore, since $F(\bar{x})$ is compact, we can apply Proposition \ref{prop:domination property} to obtain the existence of $\{\bar{y}^1,\ldots,\bar{y}^p\} \subseteq F(\bar{x})$ such that, for all $ i \in [p],$ it holds $\bar{y}^i \preceq_K y^i.$ It is then straightforward to verify that $ (\bar{x},\bar{y}^1,\ldots,\bar{y}^p ) \in \wargmin \eqref{vpp}.$ 

 Suppose now that $(\bar{x},\bar{y}^1,\ldots,\bar{y}^p ) \in \wargmin \eqref{vpp},$ and that $$\{\bar{y}^1, \ldots, \bar{y}^p\} \cap \WMin(F(\bar{x}),K) = \emptyset.$$ Then, for each $i \in [p]$ we can find $\hat{y}^i \in F(\bar{x})$ such that $\hat{y}^i \prec_K \bar{y}^i.$  However, this would imply  that $f^p (\bar{x},\hat{y}^1,\ldots,\hat{y}^p) \prec_{K^p}f^p(\bar{x},\bar{y}^1,\ldots,\bar{y}^p),$ a contradiction to the weak minimality of $(\bar{x},\bar{y}^1,\ldots,\bar{y}^p)$ for \eqref{vpp}. 

 $(ii)$ The proof of the first part of the statement is analog to that in $(i)$. Suppose next that $(\bar{x},\bar{y}^1,\ldots,\bar{y}^p ) \in \argmin \eqref{vpp}$ and that, for some $i \in [p],$ we have $\bar{y}^i \notin \Min(F(\bar{x}),K).$ Then, there exists $\hat{y}^i$ such that $\hat{y}^i \precneq_K \bar{y}^i.$ It is then possible to check that $f^p(\bar{x},\bar{y}^1,\ldots, \hat{y}^i, \ldots, \bar{y}^p) \precneq_{K^p} f^p(\bar{x},\bar{y}^1,\ldots, \bar{y}^i, \ldots, \bar{y}^p),$ a contradiction.  
\end{proof}

 Before continuing we recall that, for a bounded  set $A \subseteq \R^m$ and $\varepsilon > 0,$ the so called $\varepsilon$-external covering number of $A$  is defined by 
\begin{equation*}
    n_\varepsilon(A) := \min\left\{|B| \mid B \subseteq \R^m, \; |B| < +\infty, \; A \subseteq B + \varepsilon\mathbb{B} \right\}.
\end{equation*} This concept has proven to be a useful tool in other areas of research, including both geometric analysis and machine learning,  see for example \cite[Chapter 4]{AvidanGianMilman2015} and \cite[Chapter 27]{shwartzben-david2014}, respectively. It is well known that $n_\varepsilon$ is well defined and finite for each bounded $A$, and that it is monotone increasing with respect to inclusion, that is, $A \subseteq B$ implies $n_\varepsilon(A) \leq n_\varepsilon(B).$  

 We are now ready to state the main result of the section. It establishes that it is possible to approximate with arbitrary precision the set of (weakly) minimal solutions of \eqref{sp} by those of \eqref{vpp}.
 
\begin{Theorem}\label{thm:relations vectorization scheme}
The following relationships hold:
\begin{enumerate}
    \item \label{item:relations vectorization scheme_wmin} $\bigcup\limits_{p \in \N} \wargmin_x \eqref{vpp} \subseteq \wargmin \eqref{sp} = \bigcap\limits_{\varepsilon >0}\bigcup\limits_{p\in \N} \varepsilon\textrm{-}\wargmin_x  \eqref{vpp},$
    \item \label{item:relations vectorization scheme_min} $\bigcup\limits_{p \in \N} \argmin_x \eqref{vpp} \subseteq \argminb \eqref{sp} \subseteq \bigcap\limits_{\varepsilon >0}\bigcup\limits_{p\in \N} \varepsilon\textrm{-}\argmin_x  \eqref{vpp}.$
\end{enumerate}
\end{Theorem}
\begin{proof}
The first inclusion of \ref{item:relations vectorization scheme_wmin} is just an immediate consequence of Proposition \ref{prop:ssvp subset sssp} $\ref{item:monot epsilon}$ with $\varepsilon = 0.$ In order to see the equality in the second part, we first apply Proposition \ref{prop:ssvp subset sssp} $\ref{item:monot epsilon}$ and Proposition \ref{prop:lalita=wintersectionepsilon} \ref{item:wargmin=intersectionepsilon} to obtain

\begin{equation*}
\bigcap\limits_{\varepsilon >0}\bigcup\limits_{p\in \N} \varepsilon\textrm{-}\wargmin_x  \eqref{vpp} \subseteq \bigcap\limits_{\varepsilon >0} \varepsilon\textrm{-}\wargmin  \eqref{sp} = \wargmin \eqref{sp}.
\end{equation*} Thus, it remains to show that $\wargmin \eqref{sp} \subseteq \bigcap_{\varepsilon >0}\bigcup_{p\in \N} \varepsilon\textrm{-}\wargmin_x  \eqref{vpp}.$

Fix $\bar{x} \in \wargmin \eqref{sp}$ and let $\varepsilon >0.$ Then, because $\varepsilon e \in \Int K,$ there exists $r_\varepsilon >0$ such that $\varepsilon e +r_\varepsilon \B \subseteq \Int K.$ From this, it follows that 
\begin{equation}\label{eq:y include intK}
\forall \; y \in \R^m: y+ r_\varepsilon \B \subseteq y - \varepsilon e + \Int K.
\end{equation} Since $F(\bar{x})$ is compact, we have that $p:=n_{\frac{r_\varepsilon}{2}}(F(\bar{x})) < + \infty.$ Therefore, we can find $\{\hat{y}^1,\ldots,\hat{y}^p\} \subseteq \R^m$ such that 
\begin{equation}\label{eq:finite cover Fhat}
F(\bar{x}) \subseteq \{\hat{y}^1,\ldots,\hat{y}^p\} + \frac{r_\varepsilon}{2} \B.
\end{equation}
Now, by the definition of $p,$ it follows the existence of $\{\bar{y}^1,\ldots,\bar{y}^p\} \subseteq F(\bar{x})$ such that
\begin{equation}\label{eq:ybar def}
\forall \; i\in [p]: \bar{y}^i \in \left(\hat{y}^i +\frac{r_\varepsilon}{2} \B \right)\cap F(\bar{x}).
\end{equation} Otherwise, we can assume without loss of generality  that $\left(\hat{y}^p +\frac{r_\varepsilon}{2} \B \right)\cap F(\bar{x}) = \emptyset,$ which would imply  $n_{\frac{r_\varepsilon}{2}}(F(\bar{x})) \leq p-1,$ a contradiction.

 We now claim that $(\bar{x},\bar{y}^1,\ldots,\bar{y}^p) \in \varepsilon\textrm{-} \wargmin \eqref{vpp}.$ Indeed, otherwise we can find $(x,y^1,\ldots,y^p) \in \gph  F^p$ such that
\begin{equation}\label{eq:epsilon decrease cover}
\forall\; i \in [p]: y^i \prec_K \bar{y}^i - \varepsilon e.
\end{equation} Thus, from \eqref{eq:finite cover Fhat}, \eqref{eq:ybar def}, \eqref{eq:y include intK} and \eqref{eq:epsilon decrease cover}, we deduce that
%
\begin{eqnarray*}
F(\bar{x}) & \subseteq & F(\bar{x}) + K\\
           & \subseteq & \bigcup\limits_{i=1}^p \left(\hat{y}^i + \frac{r_\varepsilon}{2} \B\right) + K\\
           & \subseteq &\bigcup\limits_{i=1}^p \left(\bar{y}^i + r_\varepsilon \B\right) + K\\
           & \subseteq &  \bigcup\limits_{i=1}^p \left(\bar{y}^i - \varepsilon e + \Int K\right) + K\\ 
           & \subseteq & \bigcup\limits_{i=1}^p \left(y^i + \Int K\right) + K\\
           & = & \left\{y^1,\ldots, y^p\right\} + \Int K\\
           & \subseteq & F(x) + \Int K,
\end{eqnarray*} 
which implies that $F(x) \prec^\ell_K F(\bar{x}),$ a contradiction to $\bar{x} \in \wargmin \eqref{sp}.$ This proves \ref{item:relations vectorization scheme_wmin}.

The first inclusion of \ref{item:relations vectorization scheme_min} is just an immediate consequence of Proposition \ref{prop:ssvp subset sssp} $\ref{item:monot epsilon for min}$ with $\varepsilon = 0.$ The proof of the second inclusion follows the same arguments as those given for the similar inclusion in \ref{item:relations vectorization scheme_wmin}, but replacing $\wargmin \eqref{sp},$ $\varepsilon\textrm{-} \wargmin_{x} \eqref{vpp}$ and \eqref{eq:epsilon decrease cover}  by  $\argminb \eqref{sp},$ $\varepsilon\textrm{-} \argmin_{x} \eqref{vpp}$ and 
\begin{equation}\label{eq:epsilon decrease cover'}
\forall\; i \in [p]: y^i \preceq_K \bar{y}^i - \varepsilon e,
\end{equation} respectively. The statement follows.
\end{proof}

\begin{Remark}
As a consequence of the proof of Theorem \ref{thm:relations vectorization scheme} we in fact also have that 
$$ \wargmin \eqref{sp}
=
\bigcap\limits_{\varepsilon >0}\bigcup\limits_{p\in \N} \varepsilon\textrm{-}\wargmin_x  \eqref{vpp}
= 
\bigcap\limits_{\varepsilon >0}\bigcup\limits_{p\in \N} \varepsilon\textrm{-}\argmin_x  \eqref{vpp}$$
\end{Remark}

The following example shows that, in Theorem \ref{thm:relations vectorization scheme}  \ref{item:relations vectorization scheme_min}, the second inclusion is indeed strict in general. Hence, we do not obtain equality for the minimal solutions but only for the weakly minimal solutions.
\begin{Example}
\label{ex:violation vectorization equality for min}
Let $\Omega=[0,1]$, $K=\R^2_+$, $e= (1,1)^\top,$ and consider $F: \R \rightrightarrows \R^2$ given by 
$$F(x):= \left\{
\begin{array}{ll} 
\{ (0,x)^\top  
\}
& \textrm{ if } x \in \Omega, \\
\emptyset & \textrm{ if } x \notin \Omega. \end{array} \right. $$
Then, it is easy to check that $ \argminb \eqref{sp}=\{ 0 \} $. 
However, for every $\varepsilon >0$  it is $(x,(0,x)^\top)\in \varepsilon\textrm{-}\argmin \eqref{vpp}$ for $p=1$ and for all $x\in\Omega$. Hence, we have $$ \argminb \eqref{sp} =\{0\} \nsupseteq [0,1]= \bigcap\limits_{\varepsilon >0}\bigcup\limits_{p\in \N} \varepsilon\textrm{-}\argmin_x  \eqref{vpp}.$$
%
\end{Example}

\begin{Corollary} The following statements hold:
\begin{enumerate}
    \item \label{cor:sandwich_wmin} Suppose that $F\left(\wargmin \eqref{sp}\right)$ is bounded. Then, for every $\varepsilon>0$ we can find $p \in \N$ such that
\begin{displaymath}
\wargmin_x \eqref{vpp} \subseteq \wargmin \eqref{sp} \subseteq \varepsilon\textrm{-} \wargmin_x \eqref{vpp}.
\end{displaymath}
\item \label{cor:sandwich_min} Suppose that $F\left(\argminb \eqref{sp}\right)$ is bounded. Then, for every $\varepsilon>0$ we can find $p \in \N$ such that
\begin{displaymath}
\argmin_x \eqref{vpp} \subseteq \argminb \eqref{sp} \subseteq \varepsilon\textrm{-} \argmin_x \eqref{vpp}.
\end{displaymath}
\end{enumerate}

\end{Corollary}
\begin{proof}
The first inclusions are already known from Theorem \ref{thm:relations vectorization scheme}. In order to show the second inclusion of \ref{cor:sandwich_wmin}, note that from the proof of Theorem \ref{thm:relations vectorization scheme} we also get:
\begin{equation}\label{eq:pforx}
x \in \wargmin \eqref{sp} \Longrightarrow x \in \varepsilon\textrm{-}\wargmin_x (\mathcal{VP}_{p_x}),
\end{equation} where $p_x:= n_{\frac{r_\varepsilon}{2}}(F(x)),$ and $r_\varepsilon$ is such that \eqref{eq:y include intK} holds. Set now $p:= n_{\frac{r_\varepsilon}{2}}\left(F\left(\wargmin \eqref{sp}\right)\right).$ Then, since $F\left(\wargmin \eqref{sp}\right)$ is bounded, it is clear that $p \in \N.$ Furthermore, we also have that 
\begin{equation}\label{eq:suppx}
\max \left\{p_x \mid x \in \wargmin \eqref{sp} \right\}  \leq p.
\end{equation} Thus, taking into account \eqref{eq:pforx}, \eqref{eq:suppx} and Proposition \ref{prop:sset monot} \ref{item:monot p}, we deduce that 
\begin{eqnarray*}
\wargmin \eqref{sp} & \subseteq & \bigcup\limits_{x \in \wargmin \eqref{sp}} \varepsilon\textrm{-} \wargmin_x \left(\mathcal{VP}_{p_x}\right)\\
                    & \subseteq & \varepsilon \textrm{-}\wargmin_x \eqref{vpp}, 
\end{eqnarray*} as desired. The second inclusion of \ref{cor:sandwich_min} holds analogously.
\end{proof}

As Example \ref{ex:strict inclusion} in the next section will show, the first inclusions in the statements of Theorem \ref{thm:relations vectorization scheme} may be strict. In the final result of this section we show that, nevertheless, we can restrict ourselves to compute solutions of \eqref{sp} using the vectorization scheme. The reason is that, by doing that, we lose no quality sets in the image space.

\begin{Theorem}
Suppose that $\Omega$ is compact, that $F\left(\Omega\right)$ is bounded, and that $\gph F$ is closed. Then,

\begin{displaymath}
\forall\; x\in \Omega, \exists\; \bar{x} \in \cl \left(\bigcup\limits_{p \in \N} \argmin_x  \eqref{vpp}\right) : F(\bar{x})\preceq^\ell_K F(x).
\end{displaymath}
\end{Theorem}

\begin{proof}
Under our assumptions, it is easy to check that the feasible sets of the problems \eqref{vpp}, that is, $\gph F^p \cap \left(\Omega \times \prod_{i=1}^p\R^m\right)$ are compact. Thus, in particular, we have that also the sets $f^p \left( \gph F^p \cap \left(\Omega \times \prod_{i=1}^p\R^m\right)\right)$ for $p \in \N$ are compact.

Now, fix $x \in \Omega,$ together with a countable dense subset $\{y^k\}_{k\in \N}$ of $F(x).$ Suppose that $x \notin \bigcup_{p \in \N} \argmin_x  \eqref{vpp}:$ otherwise there is nothing to prove. Then,
\begin{equation}\label{eq:notin argmin vp}
\forall\; p\in \N: \left(x,y^1,\ldots,y^p \right) \notin \argmin \eqref{vpp}.
\end{equation} Therefore, for every $p \in \N,$ we can use the compactness of $f^p \left( \gph F^p \cap \left(\Omega \times \prod_{i=1}^p\R^m\right)\right),$  \eqref{eq:notin argmin vp},  and Proposition \ref{prop:domination property}, to obtain an element $\left(x^p,y^{1,p},\ldots,y^{p,p}\right) \in \argmin \eqref{vpp}$ such that
\begin{equation}\label{eq:ykplessyk}
\forall\; k \in [p]: y^{k,p} \preceq_K y^k.
\end{equation} Now, since $\Omega$ is compact, we can assume without loss of generality that the sequence $\{x^p\}_{p\in \N}$ converges to some element $\bar{x} \in \Omega.$ We claim that $F(\bar{x}) \preceq^\ell_K F(x).$ Indeed, it is clear that $\bar{x} \in  \cl \left(\bigcup_{p \in \N} \argmin_x  \eqref{vpp}\right).$ Moreover, because $F(\Omega)$ is bounded, we can find for every $k \in \N$ an accumulation point $\bar{y}^k$ of the sequence $\{y^{k,p}\}_{p\geq k}.$ Taking the limit with respect to $p$ in \eqref{eq:ykplessyk} over a suitable subsequence, we deduce that
\begin{equation}\label{eq:ybarkplessyk}
\forall\; k \in \N: \bar{y}^k \preceq_K y^k.
\end{equation} Furthermore, because $\gph F$ is closed, we have that $\{\bar{y}^k\}_{k\in \N}\subseteq F(\bar{x}).$  Suppose now that $F(\bar{x}) \npreceq^\ell_K F(x).$ Then, we can find an element $y \in F(x) \setminus \left(F(\bar{x})+K\right).$ Since $F(\bar{x})$ is compact, the set $F(\bar{x})+K$ is closed. Therefore, there exists $\varepsilon >0$ such that $\left( y + \varepsilon \B\right) \cap \left( F(\bar{x})+K \right) = \emptyset.$ However, since $\{y^k\}_{k\in \N}$ is dense in $F(x),$ we can find $y^{k_0} \in y + \varepsilon \B$ for some $k_0 \in \N.$ This implies
\begin{equation*}
y^{k_0} \in \left(y + \varepsilon \B\right) \cap \left( \bar{y}^{k_0}+K \right) \subseteq \left(y + \varepsilon \B\right) \cap \left( F(\bar{x})+K\right),
\end{equation*} a contradiction. The statement follows.
\end{proof}

\section{Finite Dimensional Vectorization}
\label{section:fdvp}

In the previous section we saw that the solutions of \eqref{vpp} can approximate with arbitrary accuracy those of \eqref{sp}. In particular, Theorem \ref{thm:relations vectorization scheme} shows that the inclusions $\wargmin_x \eqref{vpp} \subseteq \wargmin \eqref{sp}$  and  $ \argmin_x \eqref{vpp} \subseteq \argminb \eqref{sp}$ hold for any $p \in \N.$ However, ideally one would like to have equality in the above inclusions. This would give us a guarantee that, after solving the multiobjective optimization problem \eqref{vpp}, we have already computed all of the solutions of \eqref{sp} that are of interest. In this section we study that particular case. This ideal notion of equivalence between the set optimization problem and the multiobjective optimization problems is formally captured in the following definition.

\begin{Definition}
Consider problem \eqref{sp} and the parametric family \eqref{vpp} for $p \in \N.$ \begin{enumerate}
    \item We say that \eqref{sp} satisfies the weakly minimal finite dimensional vectorization  property \textup{(wFDVP)} if
\begin{equation*}
\exists\; p \in \N: \wargmin_x \eqref{vpp} = \wargmin \eqref{sp}.
\end{equation*} 
\item  We say that \eqref{sp} satisfies the minimal finite dimensional vectorization  property \textup{(mFDVP)} if
\begin{equation*}
\exists\; p \in \N: \argmin_x \eqref{vpp} = \argminb \eqref{sp}.
\end{equation*} 
\end{enumerate}
\end{Definition}

A natural question is whether the class of set optimization problems satisfying either finite dimensional vectorization property is large enough. In our first result of the section, we provide a partial answer to this question in a specific topological sense. In order to do this, we consider a metric on a particular class of set-valued mappings that has been used in the stability analysis of set optimization problems, see \cite{KarunaLalitha2020, SongTangWang2013}.  This metric is defined by means of the classical Hausdorff distance on the class of nonempty compact subsets of $\R^m,$ which is denoted by $d_H$ in this paper. Recall that, for nonempty compact sets $A,B \subseteq \R^m,$ 
$$d_H(A,B):= \inf \left\{\varepsilon \geq 0 \mid A \subseteq B + \varepsilon \B, B \subseteq A + \varepsilon \B\right\}.$$

\begin{Theorem}\label{thm:densityresult}
Consider the metric space $\left(\mathcal{F}(\Omega),d\right),$ where
$$\mathcal{F}(\Omega):= \left\{F: \R^n \rightrightarrows \R^m \mid \dom F = \Omega, \; F(\Omega) \textrm{ is bounded, } \forall\; x\in \Omega: F(x) \textrm{ is compact}\right\} $$  and $d:\mathcal{F}(\Omega) \times \mathcal{F}(\Omega) \rightarrow \R_+$ is given by $d(F_1,F_2):= \sup\limits_{x \in \Omega} d_H(F_1(x),F_2(x)).$  Furthermore, let 
 $$\mathcal{V}(\Omega):= \left\{F \in \mathcal{F}(\Omega) \mid \eqref{sp} \textrm{ satisfies \textup{(wFDVP)}}\right\}.$$ Then, $\mathcal{V}(\Omega)$ is dense in $\mathcal{F}(\Omega).$
\end{Theorem}

\begin{proof}
Fix $F \in \mathcal{F}(\Omega)$ and $\varepsilon >0.$ Next, similar to the proof of Theorem  \ref{thm:relations vectorization scheme}, for every $x \in \Omega$ we can set $p_x: = n_{\varepsilon}(F(x))$ and find $\left\{y^{x,1},\ldots, y^{x,p_x}\right\} \subseteq \R^m$ such that
\begin{equation}\label{eq:cover1}
F(x) \subseteq \left\{y^{x,1},\ldots, y^{x,p_x}\right\}+ \varepsilon \mathbb{B}.
\end{equation} Moreover, because of the minimality of $p_x,$ it is easy to check that we also have 
\begin{equation}\label{eq:cover2}
\left\{y^{x,1},\ldots, y^{x,p_x}\right\} \subseteq F(x) + \varepsilon \mathbb{B}.
\end{equation} Thus, the inclusions \eqref{eq:cover1} and \eqref{eq:cover2} together imply 
\begin{equation}\label{eq:partialhausdorff}
d_H \left(F(x), \left\{y^{x,1},\ldots, y^{x,p_x}\right\}\right) \leq \varepsilon.
\end{equation} Define now $\tilde{F}: \R^n \rightrightarrows \R^m$ as 
$$\tilde{F}(x):= \left\{
\begin{array}{ll}
      \left\{y^{x,1},\ldots, y^{x,p_x}\right\} & \textrm{ if }x \in \Omega,\\
      \emptyset & \textrm{ if } x\notin \Omega. \\
\end{array} 
\right. $$ 
Then, $\tilde{F}$ is compact valued. Moreover, \eqref{eq:cover2} together with  the fact that $F(\Omega)$ is bounded, imply that $\tilde{F}(\Omega)$ is bounded. Thus, $\tilde{F} \in \mathcal{F}(\Omega).$ Next, taking the supremum over all $x \in \Omega$ in \eqref{eq:partialhausdorff}, we obtain that $d(F, \tilde{F}) \leq \varepsilon.$ Finally, because $F(\Omega)$ is bounded, we have that $\max\limits_{x \in \Omega} p_x  < + \infty.$ The fact that $\tilde{F} \in \mathcal{V}(\Omega)$ follows from Theorem \ref{thm:fdvpfinite} below, where it is shown that, for set-valued mappings with finite image sets, the corresponding set optimization problem  satisfies $\textup{(wFDVP)}$.
\end{proof}

If we  define 
$$\mathcal{V}'(\Omega) := \left\{F \in \mathcal{F}(\Omega) \mid \eqref{sp} \textrm{ satisfies \textup{(mFDVP)}}\right\},$$  one could attempt to replace $\mathcal{V}(\Omega)$ by $\mathcal{V}'(\Omega)$ in Theorem \ref{thm:densityresult}. However, in this case a proof (if any)  would require different arguments since Theorem  \ref{thm:fdvpfinite} cannot be applied, see Example \ref{ex:violation mFDVP}.

The following example illustrates a set-valued mapping $F$ for which the associated set optimization problem exhibits a pathological behaviour. Specifically, it shows that, in general, we have $ \mathcal{F}(\Omega)\neq \mathcal{V}(\Omega),$  $ \mathcal{F}(\Omega)\neq \mathcal{V}'(\Omega),$ and the first inclusions in Theorem \ref{thm:relations vectorization scheme} \ref{item:relations vectorization scheme_wmin}, \ref{item:relations vectorization scheme_min} are strict. 

\begin{Example}\label{ex:strict inclusion}
Let $\Omega = [0,1]$, $Y=\R^2$ and $K=\R^2_+$. 
We construct a set-valued map $F\colon\R \rightrightarrows \R^2$ with $\dom F = \Omega$ which is piecewise constant and for which the image sets $F(x)$ have a finite cardinality for any $x\in[0,1)$, and where $F(1)$ has infinite many elements. We illustrate some of the sets $F(x)$ in 
Figure \ref{fig:Cantor++}, which we explain in more detail below after constructing the elements of $F(x)$.
\begin{figure}[ht]\centering
  \begin{tikzpicture}[scale= 2.6]
    \draw[thick] (2,-3.95) -- (2,-4.05) node[below]{$2$};
	\draw[thick] (3,-3.95) -- (3,-4.05) node[below]{$3$};
	\draw[thick] (4,-3.95) -- (4,-4.05) node[below]{$4$};
	\draw[thick] (5,-3.95) -- (5,-4.05) node[below]{$5$};
	
	\draw[thick] (2.05,-4) -- (1.95,-4) node[left]{$-4$};
	\draw[thick] (2.05,-3) -- (1.95,-3) node[left]{$-3$};
	\draw[thick] (2.05,-2) -- (1.95,-2) node[left]{$-2$};
	\draw[thick] (2.05,-1) -- (1.95,-1) node[left]{$-1$};
	
	\draw[thick,->] (2,-4) -- (5.5,-4) node[right]{$y_1$};
	\draw[thick,->] (2,-4) -- (2,-0.5) node[left]{$y_2$};
    
   \draw[fill=black] (2.5,-2.5) circle (1pt) node[above left] {$\bar{y}^{0,0}$};
   
   \draw[fill=green] (3,-2) circle (1pt) node[above left] {$\bar{y}^{0,1}$};
   \draw[fill=green] (4.25,-3.25) circle (1pt) node[above left] {$\bar{y}^{1,1}$};
   
   \draw[fill=blue] (3.25,-1.75) circle (1pt) node[above left] {$\bar{y}^{0,2}$};
   \draw[fill=blue] (4.5,-3) circle (1pt) node[above left] {$\bar{y}^{1,2}$};
   \draw[fill=blue] (3.875,-2.375) circle (1pt) node[above left] {$\bar{y}^{2,2}$};
   
   \draw[fill=red] (3.375,-1.625) circle (1pt) node[above left] {$\bar{y}^{0,3}$};
   \draw[fill=red] (4.625,-2.875) circle (1pt) node[above left] {$\bar{y}^{1,3}$};
   \draw[fill=red] (4,-2.25) circle (1pt) node[above left] {$\bar{y}^{2,3}$};
   \draw[fill=red] (3.6875,-1.9375) circle (1pt) node[above left] {$\bar{y}^{3,3}$};
   
   \draw (3.4375,-1.5625) node[rotate=45]{...};
   \draw (4.6875,-2.8125) node[rotate=45]{...};  
   \draw (4.0625,-2.1875) node[rotate=45]{...};
   \draw (3.750,-1.88125) node[rotate=45]{...};
   
   \draw[fill=orange] (3.5,-1.5) circle (1pt) node[above right] {$\bar{y}^{0,*}$};
   \draw[fill=orange] (4.75,-2.75) circle (1pt) node[above right] {$\bar{y}^{1,*}$};
   \draw[fill=orange] (4.125,-2.125) circle (1pt) node[above right] {$\bar{y}^{2,*}$};
   \draw[fill=orange] (3.8125,-1.8125) circle (1pt) node[above right] {$\bar{y}^{3,*}$};
   \draw[fill=orange] (3.65625,-1.65625) circle (1pt) node[above right] {$\bar{y}^{4,*}$};
   
   \draw (3.578125,-1.578125) node[rotate=-45]{...};
   
   \matrix [draw,below left] at (5.5,-0.5) {
  \node [shape=circle, draw=black, fill=black, line width=1pt,label=right:$F(0)$] {}; \\
  \node [shape=circle, draw=black, fill=green, line width=1pt,label=right:$F(\frac{1}{2})$] {}; \\
  \node [shape=circle, draw=black, fill=blue, line width=1pt,label=right:$F(\frac{2}{3})$] {}; \\
  \node [shape=circle, draw=black, fill=red, line width=1pt,label=right:$F(\frac{3}{4})$] {}; \\
  \node [shape=circle, draw=black, fill=orange, line width=1pt,label=right:$\in F(1)$] {}; \\
};
	
	\end{tikzpicture}
	\caption{Image sets $F(x)$ from Example \ref{ex:strict inclusion}}
	\label{fig:Cantor++}
\end{figure}
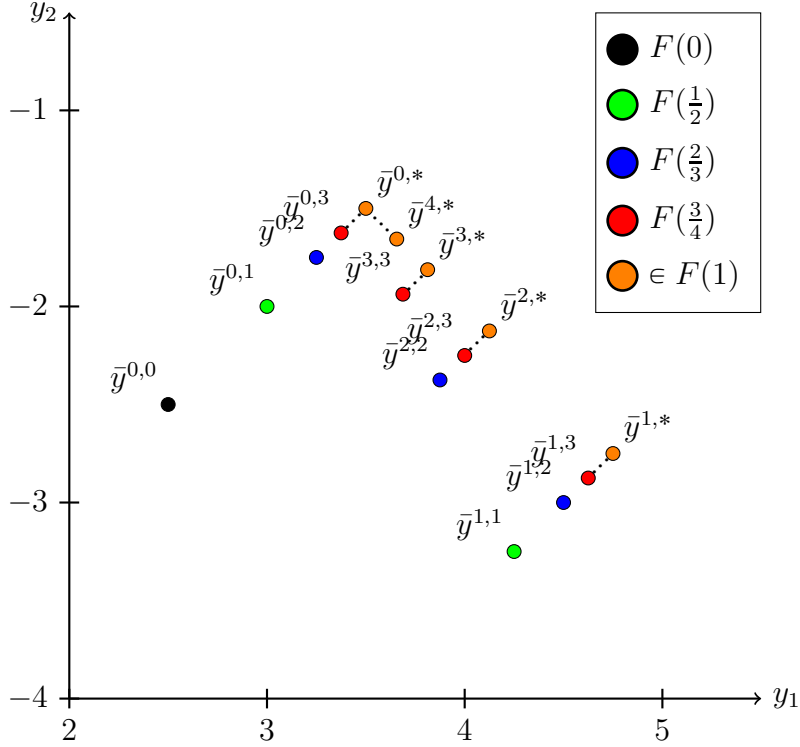

First, we define $\bar{y}^{0}:= (2.5, -2.5)^\top$ and $$\bar{y}^{i}:=\begin{pmatrix} 1+2^{-i}+2^{-i-1} \\ 1-2^{-i+1}-2^{-i}-2^{-i-1} \end{pmatrix} +\bar{y}^0 $$ for all $i\in\N$. Furthermore, for all $k\in\N \cup \{0\}$ and all $i\in\N \cup \{0\}$ with $i\leq k,$ we define $$\bar{y}^{i,k}:= \bar{y}^i+ \left(\sum_{s=i+1}^{k}2^{-s}\right)\begin{pmatrix} 1 \\ 1 \end{pmatrix}.$$
Note that this implies $\bar{y}^{i,i}=\bar{y}^i.$ 
The sets of points $\{\bar{y}^{0,k},\ldots,\bar{y}^{k,k}\}$ for $k\in\N \cup \{0\}$ will built  the image sets $F(x)$ for $x\in[0,1)$. For the image set $F(1)$ we need 
 for each $i \in \N \cup \{0\}$ the limit points of the  the sequences  $(\bar{y}^{i,k})_{k\geq i}$, i.e., 
  $$\bar{y}^{i,*}:= \lim_{k\rightarrow\infty} \bar{y}^{i,k}.$$
Next, we define $z:[0,1)\rightarrow \N \cup \{0\}$ by $z(x):= \left\lfloor \frac{1}{1-x} \right\rfloor -1$ and obtain $F\colon\R \rightrightarrows \R^2$ as
\begin{equation*}
F(x):=\begin{cases}
\bigcup\limits_{i=0}^{z(x)}\{ \bar{y}^{i,z(x)} \}  & \, \textrm{ if }x \in [0,1), \\
\bigcup\limits_{i\in\N \cup \{0\}}\{ \bar{y}^{i,*} \} & \, \textrm{ if }x=1,\\
\emptyset & \textrm{ if } x \notin \Omega.
\end{cases}
\end{equation*}
Note that by definition $F$ is piecewise constant. In fact, for every $t\in\N$ and any $x\in \left[\frac{t-1}{t},\frac{t}{t+1}\right)$, it holds
\begin{equation*}
    F(x)=F\left(\frac{t-1}{t}\right). 
\end{equation*} Figure \ref{fig:Cantor++} illustrates the sets $F(x)$ for some selected values of $x$ in $[0,1]$. Due to the fact that $F$ is piecewise constant, we only illustrate representative image sets of the first four intervals $[\frac{t-1}{t},\frac{t}{t+1})$, $t\in [4]$. We also depict the first five points $\bar{y}^{i,*}$, $i\in  [4] \cup \{0\},$ of $F(1)$. Since the set $F(1)$ consists of infinitely many nonconnected points, we cannot display the whole set. However, one can check that
\begin{equation*}
    \bar{y}^{i+1,*}=\frac{\bar{y}^{0,*}+\bar{y}^{i,*}}{2} \mbox{ holds for all $i\in \N$.}
\end{equation*}
With this equation one gets an understanding of the structure of $F(1)$.

The set-valued mapping $F$ and its corresponding problem \eqref{sp} satisfy the following properties:
\begin{enumerate}
    \item $F\in\mathcal{F}(\Omega).$
    
     This is a consequence of the definitions of $F$ and $\Omega.$
    \item $\argmin \eqref{sp}=\wargmin \eqref{sp}=\Omega.$
    
    Indeed, fix elements $v, w\in [0,1)$ and assume without loss of generality that  $k_1:=z(v) \neq z(w)=:k_2$. Otherwise, we would have $F(v)=F(w)$ and nothing would have to be shown. We assume without loss of generality $k_1+1\leq k_2$.  Now one can verify that
    \begin{equation*}
        \forall\; i\in [k_2] \cup \{0\}:\; \bar{y}^{0,k_1}_1 < \bar{y}^{i,k_2}_1 
    \end{equation*} 
    and that $\bar{y}^{0,k_1}\in F(v)$, but $\bar{y}^{0,k_1}\notin F(w)+\R^2_+$. Thus, 
    $
        F(v)\nsubseteq F(w)+\R^2_+.
    $ 
    Additionally, one can check that 
     $\bar{y}^{k_1+1,k_2}_2 < \bar{y}^{0,k_1}_2$ and 
    $\bar{y}^{k_1+1,k_2}_1 < \bar{y}^{i,k_1}_1$ holds for all $i\in [k_1]$. Thus, $\bar{y}^{k_1+1,k_2}\in F(w)$, but
     $\bar{y}^{k_1+1,k_2}\notin F(v)+\R^2_+,$
    and hence
   $
        F(w)\nsubseteq F(v)+\R^2_+.
    $ 
    
    It remains to consider the case $v\in[0,1)$ and $w=1$. Again, let $k_1=z(v)$. One can check that 
    \begin{equation*}
        \forall \; i\in\N: \;\bar{y}^{0,k_1}_1 < \bar{y}^{i,*}_1 
    \end{equation*}
     This implies 
     $F(v)\nsubseteq F(1)+\R^2_+.$
     Furthermore, it is easy to verify that $\bar{y}^{k_1+1,*}_2 < \bar{y}^{0,k_1}_2$ and $\bar{y}^{k_1+1,*}_1 < \bar{y}^{i,k_1}_1$ holds for all $i\in [k_1].$ Thus, in particular, $F(1)\nsubseteq F(v)+\R^2_+.$
      Then, we conclude that $\argmin \eqref{sp}=\Omega,$ and hence $\wargmin \eqref{sp}=\Omega$.
    
  \item \label{ex:proof wmin gap} $\bigcup_{p \in \N} \wargmin_x \eqref{vpp} \neq \wargmin \eqref{sp}$ and $\bigcup_{p \in \N} \argmin_x \eqref{vpp} \neq \argminb \eqref{sp}$.
    
    We only show the first statement for our example. The second one can be shown analogously.
    
    Since $\wargmin \eqref{sp}=\Omega $, we know that $1\in \wargmin \eqref{sp}$. Now, assume that there exists $p\in\N$ such that $1\in\wargmin_x \eqref{vpp} $. This implies the existence of  $y^1,\ldots,y^p\in F(1)$ such that $(1,y^1,\ldots,y^p)\in \wargmin \eqref{vpp}.$ 
    Then, for all $i\in [p]$ we have that  $y^i=\bar{y}^{k_i,*}$ for some $k_i\in\N \cup \{0\}$. Next, we set $j:=\max\{ k_i\mid i\in [p]\}$
    and $v:=\frac{j}{j+1}$. Then, $z(v)=j$ and $k_i\leq j$ for all $i\in [p]$ and thus $\bar{y}^{k_i,j}\in F(v)$. 
    It is now easy to check that  
     $\bar{y}^{k_i,j} \prec_{\R^2_+} \bar{y}^{k_i,*}$  for all $i\in [p]$.
    
    Hence, it is 
    $$ f^p(v,\bar{y}^{k_1,j},\ldots,\bar{y}^{k_p,j})\prec_{\R^{2p}_+} f^p(1,y^1,\ldots, y^p) .$$ This contradicts $(1,y^1,\ldots, y^p)\in \wargmin \eqref{vpp}$. Thus, 
    $ 1\in \wargmin \eqref{sp}$ but $1\not\in 
    \bigcup\limits_{p \in \N} \wargmin_x \eqref{vpp}.$ 
    
    \item $F\notin \mathcal{V}(\Omega) \cup \mathcal{V}'(\Omega).$
    
    This statement follows from \ref{ex:proof wmin gap}.
\end{enumerate}

\end{Example} 

In the rest of the section we focus on describing special classes of set optimization problems that satisfy the weakly minimal finite dimensional vectorization property $\textup{(wFDVP)}$. We start by analyzing the case in which the feasible set $\Omega$ is finite.

\begin{Theorem}
\label{thm:fdvp finite preimage}
Suppose that $|\Omega|< + \infty.$ Then,
$$
\eqref{sp} \ \mbox{  satisfies \textup{(wFDVP)}.}
$$
\end{Theorem}

\begin{proof}
We claim that $\textup{(wFDVP)}$ holds with $p := |\Omega| - 1.$ By Theorem \ref{thm:relations vectorization scheme} \ref{item:relations vectorization scheme_wmin} we already have
$\wargmin_x \eqref{vpp} \subseteq \wargmin \eqref{sp}$. For showing the remaining inclusion let $\Omega = \left\{x^1,\ldots,x^{|\Omega|}\right\}$ and suppose without loss of generality that $x^{|\Omega|} \in \wargmin \eqref{sp}.$  Then, for each $i \in [p],$ there exists $\bar{y}^i \in F\left(x^{|\Omega|}\right) \setminus \left(F(x^i) + \Int K\right).$ Furthermore, because of Proposition \ref{prop:domination property}, we can assume without loss of generality that $\bar{y}^i \in \Min\left(F\left(x^{|\Omega|}\right),K\right)$ for each $i \in [p].$  We then derive that 
 $(x^{|\Omega|},\bar{y}^1,\ldots, \bar{y}^p) \in \wargmin \eqref{vpp}.$ Indeed, otherwise  there would  exist    $(x^k,y^1,\ldots,y^p)\in \gph F^p$ for some $k\in[p]$ with $y^i \prec_K \bar{y}^i$ for all $i\in[p]$. Thus, for $i=k$ we get $\bar{y}^k\in F(x^k)+ \Int K,$ which contradicts our choice of $\bar{y}^k$. 
\end{proof}

The next class of set optimization problems that we study are those  in which the values of the set-valued objective mapping have finite cardinality. Problems with this structure arise in the study of set-based robustness concepts for  multiobjective optimization problems under uncertainty when the cardinality of the uncertainty set is finite, see \cite{idekobis2014}.

\begin{Theorem}\label{thm:fdvpfinite}
Suppose that $\max\limits_{x\in \Omega} |\Min(F(x),K)|< + \infty.$ Then,
$$
\eqref{sp}\ \mbox{ satisfies \textup{(wFDVP)}.}
$$
\end{Theorem}
\begin{proof}
Set $p:= \max\limits_{x\in \Omega} |\Min(F(x),K)|.$ According to our assumption, we have that $p\in \N.$ By Theorem \ref{thm:relations vectorization scheme} \ref{item:relations vectorization scheme_wmin}, it remains to show $ \wargmin \eqref{sp}\subseteq \wargmin_x \eqref{vpp}$. For that, fix  $\bar{x} \in \wargmin \eqref{sp}.$ Then, there exists $\bar{p} \leq p$  and elements $\bar{y}^1,\ldots,\bar{y}^{\bar{p}} \in \R^m$ such that
$$\Min(F(\bar{x}),K) = \left\{\bar{y}^1,\ldots,\bar{y}^{\bar{p}} \right\}.$$ We now claim that $(\bar{x},\bar{y}^1,\ldots,\bar{y}^{\bar{p}} ) \in \wargmin (\mathcal{VP}_{\bar{p}}).$ Indeed, otherwise we can find $(x,y^1,\ldots,y^{\bar{p}}) \in \gph F^{\bar{p}}$ such that 
\begin{equation}\label{eq:yilessbaryi}
\forall \; i \in [\bar{p}]: y^i \prec_K \bar{y}^i.
\end{equation}
Taking into account Proposition \ref{prop:domination property} and \eqref{eq:yilessbaryi}, we then find that
$$
F(\bar{x})   \subseteq    \left\{\bar{y}^1,\ldots,\bar{y}^{\bar{p}} \right\}+ K 
            \subseteq   \left\{y^1,\ldots, y^{\bar{p}} \right\} + \Int K\\
             \subseteq   F(x)+ \Int K,     $$
%
a contradiction to the fact that $\bar{x} \in \wargmin \eqref{sp}.$ Therefore, our claim is true and, in particular, we get that $\bar{x} \in \wargmin_x (\mathcal{VP}_{\bar{p}}).$ Finally, according to Proposition \ref{prop:sset monot} \ref{item:monot p}, this implies 
that $\bar{x} \in \wargmin_x \eqref{vpp}.$ 
\end{proof} The set optimization problems studied in Theorem \ref{thm:fdvp finite preimage} and Theorem \ref{thm:fdvpfinite} satisfy the
weakly minimal finite dimensional vectorization property $\textup{(wFDVP)}$. The next example shows that they do in general not also satisfy the  minimal finite dimensional vectorization property $\textup{(mFDVP)}$.
\begin{Example}\label{ex:violation mFDVP}
Let $\Omega=\{0,1,2\}$, $K=\R^2_+,$ and consider $F: \R\rightrightarrows \R^2$ given by 
$$F(x):=\begin{cases}
\{(2,0)^\top,(0,2)^\top\} &  \textrm{ if } x=0, \\ 
\{(1,-1)^\top, (0,2)^\top\} & \textrm{ if } x=1, \\
\{(2,0)^\top,(-1,1)^\top\} & \textrm{ if } x=2,\\
\emptyset & \textrm{ otherwise.}
\end{cases} $$
%
Then, it is easy to check that $ \argminb \eqref{sp}=\Omega $ and that the assumptions of Theorems \ref{thm:fdvp finite preimage} and \ref{thm:fdvpfinite} are fulfilled. Now, assume that \textup{(mFDVP)} holds. Then, there exists $p\in\N$ such that  $0\in\argminb \eqref{sp} = \argmin_x \eqref{vpp}$ holds. Therefore, there exist elements $\bar{y}^i\in F(0)$  for $i\in[p]$ such that $(0,\bar{y}^1,\ldots,\bar{y}^p)\in \argmin \eqref{vpp}$. If $\bar{y}^i= (2,0)^\top$ for all $i\in[p]$, we choose $x:=1$ and $y^i:= (1,-1)^\top$  for all $i\in[p]$. Else, we choose $x:=2$ and for all $i\in[p]$
$$y^i:=
\begin{cases}
 (2,0)^\top & \textrm{ if }\  \bar{y}^i=(2,0)^\top, \\
(-1,1)^\top  & \textrm{ if }\ \bar{y}^i= (0,2)^\top.
\end{cases}
%
$$
Then, in both cases we have $(x,y^1,\ldots,y^p)\in \gph F^p$ and 
$$f^p(x,y^1,\ldots,y^p) \precneq_{K^p} f^p(0,\bar{y}^1,\ldots,\bar{y}^p), $$ which contradicts $(0,\bar{y}^1,\ldots,\bar{y}^p)\in \argmin \eqref{vpp}$. Therefore, $F$ does not satisfy \textup{(mFDVP)}.
\end{Example}

In the next theorem, for a convex set $A \subseteq \R^m,$ we denote by $\textup{ext}(A)$  the set of extreme points of $A.$

\begin{Theorem}
\label{thm:fdvppolytope}
Suppose that $F$ only takes polyhedral values  and that $\max\limits_{x\in \Omega} |\textup{ext}(F(x))|< + \infty.$ 
Then,
$$
\eqref{sp}\ \mbox{satisfies } \textup{(wFDVP)}. 
$$

\end{Theorem}

\begin{proof}
Similarly to the proof of Theorem \ref{thm:fdvpfinite}, we set $p:= \max\limits_{x\in \Omega} |\textup{ext}(F(x))|.$ Then, our assumption implies that $p\in \N.$ 
By Theorem \ref{thm:relations vectorization scheme} \ref{item:relations vectorization scheme_wmin}, it remains to show 
$ \wargmin \eqref{sp}\subseteq \wargmin_x \eqref{vpp}$. For that, fix
 $\bar{x} \in \wargmin \eqref{sp}.$ Then, we can find $\bar{p} \leq p$  and elements $\bar{y}^1,\ldots,\bar{y}^{\bar{p}} \in \R^m$ such that
$$\textup{ext}(F(\bar{x})) = \left\{\bar{y}^1,\ldots,\bar{y}^{\bar{p}} \right\}.$$ We now claim that $(\bar{x},\bar{y}^1,\ldots,\bar{y}^{\bar{p}} ) \in \wargmin (\mathcal{VP}_{\bar{p}}).$ Indeed, otherwise there exists $(x,y^1,\ldots,y^{\bar{p}}) \in \gph F^{\bar{p}}$ such that 
\begin{equation}\label{eq:yilessbaryi2}
\forall \; i \in [\bar{p}]: y^i \prec_K \bar{y}^i.
\end{equation} Take now $\bar{y} \in F(\bar{x}).$ Because $F(\bar{x})$ is a bounded polyhedron, we have the existence of $\lambda \in \R_+^{\bar{p}}$ such that $\sum_{i=1}^{\bar{p}} \lambda_i = 1$ and $\bar{y} = \sum_{i=1}^{\bar{p}} \lambda_i \bar{y}^i.$ It then follows from \eqref{eq:yilessbaryi2} that 
$$
\bar{y}   \in   \sum\limits_{i=1}^{\bar{p}} \lambda_i y^i+ \Int K 
          \subseteq  \conv \left\{y^1,\ldots, y^{\bar{p}} \right\} + \Int K 
          \subseteq   F(x) + \Int K.      
$$ 
%
Since $\bar{y}$ was arbitrarily chosen in $F(\bar{x}),$  we derive $F(\bar x)\subseteq F(x)+\Int K$ which is a contradiction to the fact that $\bar{x} \in \wargmin \eqref{sp}.$ Therefore  $\bar{x} \in \wargmin_{x} (\mathcal{VP}_{\bar{p}}),$ and the statement of the theorem follows from Proposition \ref{prop:sset monot} \ref{item:monot p}.
\end{proof}

The following example shows that for this class $\textup{(mFDVP)}$ is not fulfilled either.
\begin{Example} We use the same definitions as in Example \ref{ex:violation mFDVP} and study the set-valued mapping
$$\tilde{F}(x):=\conv F(x). $$
Then, $\tilde{F}$ takes only bounded  polyhedral values and the assumptions of Theorem \ref{thm:fdvppolytope} are fulfilled. For $\tilde{F}$ it is also easy to check that $ \argminb \eqref{sp}=\Omega $. Now, assume that \textup{(mFDVP)} holds. Then, there exists $p\in\N$ such that  
$0\in\argminb \eqref{sp} = \argmin_x \eqref{vpp}$ 
holds. Therefore, there exist elements $\bar{y}^i\in \tilde{F}(0)$ such that $(0,\bar{y}^1,\ldots,\bar{y}^p)\in \argmin \eqref{vpp}$. If $\bar{y}^i=(2,0)^\top$ for all $i\in[p]$, we choose $x:=1$ and 
$y^i:=(1,-1)^\top 
\in \tilde{F}(1)$ for all $i\in[p]$.
Else, we choose $x:=2$, and since
$\bar{y}^i \in \tilde{F}(0)= 
\conv  \left\{(2,0)^\top, (0,2)^\top \right\}$ we have
$$ \bar{y}^i= \begin{pmatrix} 2-2\lambda_i \\ 2\lambda_i \end{pmatrix}$$ for some $\lambda_i\in [0,1]$ for all $i\in [p]$. Now, we choose for $i\in[p]$
$$y^i:= \begin{pmatrix} 2-2\lambda_i \\ \frac{2}{3}\lambda_i \end{pmatrix} \in \tilde{F}(2).$$
Then, in both cases we have $(x,y^1,\ldots,y^p)\in \gph \tilde{F}^p$ and 
$$f^p(x,y^1,\ldots,y^p) \precneq_{K^p} f^p(0,\bar{y}^1,\ldots,\bar{y}^p), $$ which contradicts $(0,\bar{y}^1,\ldots,\bar{y}^p)\in \argmin \eqref{vpp}$. Therefore, $\tilde{F}$ does not satisfy \textup{(mFDVP)}.
\end{Example}

Now we proceed to the final result of the section. It establishes that convex set optimization problems are indeed equivalent to a finite dimensional convex multiobjective optimization problem.  For proving our result  we need to recall two additional concepts. The first of these is that of the normal cone to a nonempty convex set $\Omega \subseteq \R^n$ at some point $\bar{x} \in \Omega.$ Formally, this is the set 
$$N(\bar{x},\Omega) := \{u\in  \R^n\mid \forall\; x\in \Omega:  u^\top (x-\bar{x})\leq 0\}.$$  The second concept we introduce is that of  the coderivative of a set-valued mapping with a convex graph:
\begin{Definition} \label{defcord}(\cite[Definition 2.56]
{MordukhovichNam2014path}) 
Let $F:\R^n \rightrightarrows \R^m$ be a set-valued mapping with convex graph and let $(\bar{x}, \bar{y}) \in \gph F.$ The coderivative of $F$ at $(\bar{x},\bar{y})$ is the set-valued mapping $D^* F(\bar{x},\bar{y}):\R^m \rightrightarrows \R^n $ with the values
\begin{equation*}
D^{*}F(\bar{x},\bar{y})(v)=\big\{u\in \R^n\mid(u,-v)\in N\big((\bar{x},\bar{y}),\gph F\big)\}.
\label{cod}%
\end{equation*}
\end{Definition} 
Explicit calculations of the coderivative in some particular settings can be found, for instance, in \cite[Section 2.6]
{MordukhovichNam2014path}.

Next, as preparation for the proof of the final result 
, we present a sequence of lemmata. The first of these is the classic normal cone intersection formula.
\begin{Lemma}(\cite[Corollary 2.19]{MordukhovichNam2014path})\label{thm:ncif}
Let $\Omega_i \subseteq \R^n$ for $i \in [p]$ be nonempty convex sets and let $\bar{x} \in \bigcap
_{i = 1}^p \Omega_i.$ Furthermore, assume the validity of the following constraint qualification condition:
$$\left[\sum_{i = 1}^p u^i = 0,\quad u^i \in N(\bar{x}, \Omega_i)\right] \Longrightarrow \left[ \forall \; i \in [p]:\ u^i = 0\right].$$ Then, the following intersection formula holds:
$$N\left(\bar{x}, \bigcap_{i = 1}^p \Omega_i\right) = \sum_{i = 1}^p N(\bar{x},\Omega_i),$$
where the sum on the right hand side is the Minkowski sum.
\end{Lemma}

\begin{Lemma}\label{lem:ncintGi}
Suppose that $\Omega \subseteq \Int \dom F $, that $ \gph F$ is convex, and that for some $p\in \N$ an element  $ (\bar{x},\bar{y}^1,\ldots,\bar{y}^p ) \in \gph F^p$ with  $\bar{x} \in \Omega$ is given. Then, $\gph F^p$ is convex and 
\begin{equation*}
N\left((\bar{x},\bar{y}^1,\ldots,\bar{y}^p\right),\gph F^p) =   \sum\limits_{i=1}^p \left\{ \left(u,0,\ldots,\smash[b] {\smallunderbrace{v}_{\mathclap{(i + 1)^{th} \textrm{ term}} }},\ldots,0\right) \mid \left(u,v\right) \in N\left((\bar{x},\bar{y}^i\right),\gph F)\right\}.
\end{equation*}
\end{Lemma}
\begin{proof}
Define, for $i \in [p],$ the set 
\begin{equation*}
G_i:= \left\{(x,y^1,\ldots,y^p) \in \R^n\times \prod\limits_{j=1}^p \R^m  \mid (x,y^i) \in \gph F\right\}.
\end{equation*} Then, it is easy to check that each $G_i$ is convex and that $\gph F^p= \bigcap_{i=1}^p G_i.$ Thus, in particular, the set $\gph F^p$ is convex. The statement will be a consequence of the normal cone intersection formula in Lemma \ref{thm:ncif}. In order to apply this result, we need to check the following constraint qualification:
\begin{equation}\label{eq:cq}
\left[\sum\limits_{i=1}^p w^i = 0,\; w^i \in N\left( (\bar{x},\bar{y}^1,\ldots,\bar{y}^p ),G_i\right) \right] \Longrightarrow \left[\forall\; i\in [p]: w^i = 0\right].
\end{equation} We proceed to show that \eqref{eq:cq} holds. Indeed, it is easy to verify that 
\begin{equation}\label{eq:NconeGi}
 N\left( (\bar{x},\bar{y}^1,\ldots,\bar{y}^p ),G_i\right) = \left\{ \left(u,0,\ldots,\smash[b]{\smallunderbrace{v}_{\mathclap{(i + 1)^{th} \textrm{ term}} }},\ldots,0\right) \mid \left(u,v\right) \in N\left( (\bar{x},\bar{y}^i), \gph F \right)\right\}.
\end{equation} \vspace{1pt}

For $i \in [p],$ take now $w^i \in N\left((\bar{x},\bar{y}^1,\ldots,\bar{y}^p),G_i\right)$ and suppose that $\sum_{i=1}^p w^i = 0.$ Then, according to \eqref{eq:NconeGi}, we can find $\left(u^i,v^i\right) \in N\left( (\bar{x},\bar{y}^i ),\gph F\right)$ such that  $$w^i = \left(u^i,
0,\ldots,\smash[b]{\smallunderbrace{v^{i}}_{\mathclap{(i + 1)^{th} \textrm{ term}} }},\ldots,0\right).$$ \vspace{1pt} 

Thus, from $\sum_{i=1}^p w^i = 0,$ we deduce that $v^i = 0$ for each $i\in [p]$ and that $\sum_{i=1}^p u^i = 0.$ In particular, we get that 
\begin{equation}\label{eq:morducriteria}
\forall\; i\in [p]: (u^i,0) \in N\left( (\bar{x},\bar{y}^i ),\gph F\right).
\end{equation} 
Since $\bar{x} \in \Omega \subseteq \Int \dom F ,$ there exists a neighborhood of $\bar{x}$ on which $F$ has  nonempty values. Taking this into account, together with the convexity of $\gph F,$ it is straightforward  to verify that \eqref{eq:morducriteria} implies that $u^i = 0$ for every $i \in [p].$ Thus, we find that  for all $i\in[p ]$ it holds $w^i=0$, and so \eqref{eq:cq} holds. 

Applying now Lemma \ref{thm:ncif}, we obtain that 
\begin{equation}\label{eq:ncifapplied}
N\left( (\bar{x},\bar{y}^1,\ldots,\bar{y}^p ), \bigcap\limits_{i=1}^p G_i\right) =   \sum\limits_{i=1}^p N\left( (\bar{x},\bar{y}^1,\ldots,\bar{y}^p ),G_i\right).
\end{equation} 
The statement of the lemma follows then from \eqref{eq:NconeGi} and \eqref{eq:ncifapplied}.
\end{proof}

\begin{Lemma}\label{lem:ncintgphFomega}
Suppose that $\Omega$ is convex and satisfies $\Omega \subseteq \Int \dom F.$ Furthermore, assume that $\gph F$ is convex and that, for some $p\in \N,$ an element $ (\bar{x},\bar{y}^1,\ldots,\bar{y}^p ) \in \gph F^p$ with  $\bar{x} \in \Omega$ is given. Then, the set $\gph F^p \cap \left(\Omega \times \prod_{i=1}^p\R^m\right)$ is convex and
\begin{equation*}
N\left( (\bar{x},\bar{y}^1,\ldots,\bar{y}^p ),\gph F^p \cap \left(\Omega \times \prod\limits_{i=1}^p\R^m\right)\right) = N\left( (\bar{x},\bar{y}^1,\ldots,\bar{y}^p ),\gph F^p\right)  +  N\left(\bar{x},\Omega\right) \times \{0\}.
\end{equation*}
\end{Lemma}
\begin{proof}
By Lemma \ref{lem:ncintGi}, we have that $\gph F^p$ is convex. Furthermore, since $\Omega$ is convex, it is then easy to see that $\gph F^p \cap \left(\Omega \times \prod_{i=1}^p\R^m\right)$ is convex.  For the computation of the normal cone, we will apply again the normal cone intersection formula from Lemma \ref{thm:ncif}. In order to do this, we first need to verify that the constraint qualification 
\begin{equation}\label{eq:cqfeasiblesetfp}
 N\left( (\bar{x},\bar{y}^1,\ldots,\bar{y}^p ),\gph F^p\right)\cap \left[- N\left( (\bar{x},\bar{y}^1,\ldots,\bar{y}^p ),\Omega \times \prod\limits_{i=1}^p \R^m\right)\right] = \{0\}
\end{equation} holds. Indeed, let $(u, v^1,\ldots,v^p)$ belong to that intersection. Then, taking into account that 
\begin{equation}\label{eq:nconeomegatimes0}
N\left( (\bar{x},\bar{y}^1,\ldots,\bar{y}^p ),\Omega \times \prod\limits_{i=1}^p \R^m\right) = N\left(\bar{x},\Omega\right) \times \{0\},  
\end{equation} we obtain that $u \in -N\left(\bar{x},\Omega\right) $ and $v^1,\ldots,v^p = 0.$ Thus, we get
\begin{equation*}
(u, 0,\ldots,0) \in  N\left( (\bar{x},\bar{y}^1,\ldots,\bar{y}^p ),\gph F^p\right).
\end{equation*} Similarly to the proof of Lemma \ref{lem:ncintGi}, because $\gph F^p$ is convex and $\Omega \subseteq \Int  \dom F,$ we obtain $u=0.$ This shows that \eqref{eq:cqfeasiblesetfp} holds. The statement of the lemma is then a consequence of Lemma \ref{thm:ncif} and \eqref{eq:nconeomegatimes0}.
\end{proof}

The ideas behind our next lemmata were introduced by the second author in his PhD thesis, see Step 1 and Step 2 in the proof of \cite[Theorem 4.6.1]{ernestdiss}. In the following, we need the notion of the epigraphical multifunction of $F.$ This is the set-valued mapping $\mathcal{E}_F: \R^n \rightrightarrows \R^m$ defined as $$\mathcal{E}_F(x):= F(x) + K.$$
\begin{Lemma}\label{lem:coderivatives equal}
Suppose that $\gph F$ is convex. Then, so is $\gph \mathcal{E}_F,$ and  for every point $(\bar x,\bar y)\in \gph F$ it holds
\begin{equation*}
\forall\;  v \in K^*: D^* \mathcal{E}_F(\bar{x},\bar{y})(v) = D^*F(\bar{x},\bar{y})(v).
\end{equation*}
\end{Lemma}
\begin{proof}
From the definition of $\mathcal{E}_F$ we have that 
\begin{equation}\label{eq:epi F = gph F+cone}
\gph \mathcal{E}_F = \gph F + (\{0\}\times K).
\end{equation}
Thus, in particular, $\gph \mathcal{E}_F$ is also convex. Taking into account the definition of the coderivative, we see that the statement of the lemma  holds in case we can show 
\begin{equation}\label{eq: normal cones to graph and epi}
N\big((\bar{x},\bar{y}),\gph \mathcal{E}_F\big) = N\big((\bar{x},\bar{y}),\gph F\big)\cap (\R^m \times -K^*).
\end{equation}  
We now proceed to show that \eqref{eq: normal cones to graph and epi} holds. Suppose first that $(u,v) \in N\big((\bar{x},\bar{y}),\gph \mathcal{E}_F\big).$ Then, according to \eqref{eq:epi F = gph F+cone}, this is the same as 
\begin{equation}\label{eq: normal auxiliar 1}
\forall\; (x,y) \in \gph F,\ k \in K:   u^\top(x-\bar{x}) + v^\top (y+k-\bar{y}) \leq 0.
\end{equation} Taking into account that $\gph F$ is also convex and that $0\in K$, we can put $k = 0$  in \eqref{eq: normal auxiliar 1} to obtain that $(u,v) \in N\big((\bar{x},\bar{y}),\gph F\big).$ On the other hand, by substituting $x= \bar{x},y = \bar{y}$ in \eqref{eq: normal auxiliar 1}, we get $v^\top k \leq 0$ for every $k \in K.$ According to the definition of the dual cone, this implies that $v \in -K^*.$ Hence, 
$$N\big((\bar{x},\bar{y}),\gph \mathcal{E}_F\big) \subseteq N\big((\bar{x},\bar{y}),\gph F\big)\cap (\R^m \times -K^*).$$ In order to see the reverse inclusion, choose $(u,v) \in N\big((\bar{x},\bar{y}),\gph F\big)\cap (\R^m \times -K^*).$ Then, $v \in -K^*$ and 
$$\forall\; (x,y) \in \gph F:   u^\top(x-\bar{x}) + v^\top (y-\bar{y}) \leq 0.$$ It is then easy to see that this implies \eqref{eq: normal auxiliar 1}, and  together with  the convexity of  $\gph \mathcal{E}_F$  the statement follows.
\end{proof}

\begin{Lemma}\label{lem:Acompact}
Suppose that $\Omega \subseteq \Int \dom F $ and that $\gph F$ is convex. Furthermore, consider an element $\bar{x} \in \Omega$ together with the set 
\begin{equation}\label{eq:Bdef}
B:= \left\{v \in K^* \mid v^\top e  = 1\right\}.
\end{equation}
Then, the set 
\begin{equation}\label{eq:Adef}
A:= \bigcup_{\bar{y} \in \Min(F(\bar{x}),K)} D^*F(\bar{x},\bar{y})\left(B\right) 
\end{equation}
 is compact.
\end{Lemma}
\begin{proof}
It suffices to show that $A$ is both closed and bounded. In order to see the closedness of $A,$ let $\{u^k\}_{k\in \N} \subseteq A$ be such that $u^k \to \bar{u}.$ Hence, taking into account the definition of the coderivative, there are sequences $\{y^k\}_{k\in \N} \subseteq \Min(F(\bar{x}),K)$ and $\{v^k\}_{k\in \N} \subseteq B$ such that 
\begin{equation}\label{eq:bounded auxiliar 1}
\left(u^k,-v^k\right) \in N\left((\bar{x},y^k),\gph F\right).
\end{equation} Since $F(\bar{x})$ is compact, we can assume without loss of generality that $y^k \to \hat{y} \in F(\bar{x}).$ On the other hand, from \cite[Lemma 2.4]{DureaTammer2009} it follows that the set $B$ is compact. Thus, without loss of generality we can also assume that  $v^k \to \bar{v} \in B.$ Furthermore, it is well known that the set-valued mapping $N(\cdot, \gph F)$ is closed at any point. Therefore, taking the limit when $k \to +\infty$ in \eqref{eq:bounded auxiliar 1}, we obtain 
$$(\bar{u},-\bar{v}) \in N\big((\bar{x},\hat{y}),\gph F\big),$$ or equivalently, $\bar{u} \in D^*F(\bar{x},\hat{y})(\bar{v}).$ Since $F(\bar{x})$ is compact, we can apply Proposition \ref{prop:domination property} to obtain an element $\bar{y} \in \Min(F(\bar{x}),K)$ such that $\bar{y}\preceq_K \hat{y}.$ Next, because $\bar{y} - \hat{y} \in -K,$ we have $\bar{v}^\top (\bar{y} - \hat{y}) \leq 0.$ Then, for every $(x,y) \in \gph F,$ we find that
$$\bar{u}^\top (x - \bar{x}) \leq \bar{v}^\top (y - \hat{y}) = \bar{v}^\top (y- \bar{y}) + \bar{v}^\top (\bar{y} - \hat{y}) \leq \bar{v}^\top (y - \bar{y}).$$ Thus,  $\bar{u} \in D^*F(\bar{x},\bar{y})(\bar{v})$ and the closedness of $A$ follows.

Suppose now that $A$ is not bounded. Then, we can find an unbounded sequence $\{u^k\}_{k\in \N} \subseteq A.$ Let  $\{y^k\}_{k\in \N} \subseteq \Min(F(\bar{x}),K)$ and $\{v^k\}_{k\in \N} \subseteq B$ be the sequences that satisfy \eqref{eq:bounded auxiliar 1}, and without loss of generality let us assume that $y^k \to \bar{y} \in F(\bar{x})$ and $v^k \to\bar{v} \in B$. Then, we also have
\begin{equation}\label{eq:bounded auxiliar 2}
\left(\frac{u^k}{\|u^k\|},-\frac{v^k}{\|u^k\|}\right) \in N\big((\bar{x},y^k),\gph F\big).
\end{equation} Without loss of generality we can now assume that the sequence $\left\{\frac{u^k}{\|u^k\|} \right\}_{k \in \N} $ converges to some element $\bar{u} \in \R^n$ \ with $\|\bar{u}\|=1$. Then, taking the limit when $k \to + \infty$ in \eqref{eq:bounded auxiliar 2}, we get, again due to the closedness of the set-valued mapping $N(\cdot, \gph F)$,  
\begin{equation}\label{eq:morducrit2}
(\bar{u},0) \in N\big((\bar{x},\bar{y}),\gph F\big).
\end{equation} Similarly to the proof of Lemma \ref{lem:ncintGi}, since $F$ is well defined in a neighborhood of $\bar{x},$ the inclusion \eqref{eq:morducrit2} implies that $\bar{u} = 0.$ This contradicts  the fact that $\|\bar{u}\|=1.$
\end{proof}

\begin{Lemma}\label{lemm:oc}
Suppose that $\Omega$ is convex and satisfies $\Omega \subseteq \Int \dom F.$ Furthermore, assume that $\gph F$ is convex and that $F$ is locally bounded at a point  $\bar{x} \in \Omega.$ Then, 
\begin{equation*}
\bar{x} \in \wargmin \eqref{sp} \Longleftrightarrow 0 \in \cl \conv\left(\bigcup_{\bar{y} \in \Min(F(\bar{x}),K)} D^* F(\bar{x},\bar{y})\left( B\right)\right) + N(\bar{x},\Omega), 
\end{equation*} 
where  $B$ is given in \eqref{eq:Bdef}.
\end{Lemma} 

\begin{proof} The statement will be a consequence of \cite[Theorem 6.5 $(i)$]{bouzaquintanatuantammer2020}, where it is shown that
\begin{equation*}
\bar{x} \in \wargmin \eqref{sp} \Longleftrightarrow 0 \in \cl \conv\left(\bigcup_{\bar{y} \in \Min(F(\bar{x}),K)} D^* \mathcal{E}_F(\bar{x},\bar{y})\left( \partial \psi_e (0)\right)\right) + N(\bar{x},\Omega),
\end{equation*}
with $\partial$ denoting the Fenchel subdifferential from convex analysis and $\psi_e$  denoting the so called Tammer functional with respect to $e$, see \cite{DureaTammer2009}. In the following we justify that all of the assumptions of this theorem are fulfilled. Indeed, it can be verified that our concept of weakly minimal solutions of \eqref{sp} is equivalent to the one in \cite[Definition 6.1 $(i)$]{bouzaquintanatuantammer2020}   because $F(x)$ is compact for every $x \in \Omega.$ The set-valued mapping $F$ is $\preceq^{(l)}_K$-convex (see \cite[Definition 2.4]{bouzaquintanatuantammer2020}) since $\gph \mathcal{E}_F$ is convex, see \cite[Remark 2.5]{bouzaquintanatuantammer2020}. Furthermore, $F$ is  $l$-bounded at $\bar{x}$ (see \cite[Definition 2.6 $(v)$]{bouzaquintanatuantammer2020}) because $F$ is locally bounded at $\bar{x}$. Finally, the values of $F$ are strongly $K$-compact (see \cite[Definition 4.5]{bouzaquintanatuantammer2020}) because they are in particular compact. The remaining assumptions are common.

From the equivalence established in the theorem it suffices to show that for each $\bar{y} \in F(\bar{x})$ we have $D^*\mathcal{E}_F(\bar{x},\bar{y})\left(\partial \psi_e (0)\right) = D^*F(\bar{x},\bar{y})(B).$ But this is a consequence of \cite[Lemma 2.4]{DureaTammer2009} and  Lemma \ref{lem:coderivatives equal}, where it is shown that $\partial \psi_e (0) = B$ and that $D^*\mathcal{E}_F(\bar{x},\bar{y}) =D^*F(\bar{x},\bar{y}),$ respectively. The statement follows. 
\end{proof}

We are now ready to establish the main theorem.
\begin{Theorem}
Suppose that $\Omega$ is convex and satisfies $\Omega \subseteq \Int \dom F.$ Furthermore, assume that $\gph F$ is convex, and that $F$ is locally bounded at any point in $\Omega.$ Then,
$$
 \eqref{sp} \mbox{ satisfies \textup{(wFDVP)} with }\ p:= n+1.
$$
\end{Theorem}
\begin{proof} According to Theorem \ref{thm:relations vectorization scheme}, we just need to prove that, with $p:= n+1,$ 
 $$\wargmin \eqref{sp} \subseteq \wargmin_x \eqref{vpp}.$$ In order to see this, fix $\bar{x} \in \wargmin \eqref{sp}.$ Then, from Lemma \ref{lemm:oc} we have
\begin{equation}\label{eq:opt conf convex l}
0 \in \cl \conv A + N(\bar{x},\Omega), 
\end{equation} where $A$ is defined in \eqref{eq:Adef}. Now, according to Lemma \ref{lem:Acompact}, the set $A$ is compact. Thus, so is the set $\conv A.$ In particular, this implies that $\cl \conv A = \conv A.$ This fact, together with Caratheodory's theorem \cite[Theorem 2.1.6]{bazaraa2006}, allow us to deduce that \eqref{eq:opt conf convex l} \textrm{ holds} if and only if 
\begin{equation*}
    \begin{cases}
\exists \left\{\bar{y}^1,\ldots \bar{y}^{n+1}\right\} \subseteq \Min(F(\bar{x}),K),\; \left\{v^1,\ldots,v^{n+1}\right\} \subseteq B, \lambda \in \R^{n+1}_+: \\
\text{\textbullet} \;\;0 \in \sum\limits_{i=1}^{n+1} \lambda_iD^*F(\bar{x},\bar{y}^i)(v^i) +N(\bar{x},\Omega),\\
\text{\textbullet} \;\; \sum\limits_{i=1}^{n+1} \lambda_i = 1,
\end{cases}
\end{equation*}
or equivalently, 
\begin{equation*}
    \begin{cases}
\exists \left\{\bar{y}^1,\ldots \bar{y}^{n+1}\right\} \subseteq \Min(F(\bar{x}),K),\; \left\{v^1,\ldots,v^{n+1}\right\} \subseteq B, \;\left\{u^1,\ldots, u^{n+1}\right\} \subseteq \R^n, \lambda \in \R^{n+1}_+: \\
\text{\textbullet} \;\;0 \in \sum\limits_{i=1}^{n+1} \lambda_i u^i + N(\bar{x},\Omega),\\
\text{\textbullet} \;\;\forall\; i \in [n+1]: (u^i,-v^i) \in N\left((\bar{x},\bar{y}^i), \gph F \right),\\
\text{\textbullet} \;\; \sum\limits_{i=1}^{n+1} \lambda_i = 1.
\end{cases}
\end{equation*} 
Next, we set
 $\bar{u}^i:= \lambda_i u^i $ and $\bar{v}^i:= \lambda_i v^i$ for $i \in [n+1].$ 
 As there exists $i\in[n+1]$ with $\lambda_i\neq 0$ and as $v^j\in B$ and thus $v^j \not=0$ for all $j\in[n+1]$,
 the above implies
\begin{equation*}
    \begin{cases}
\exists \left\{\bar{y}^1,\ldots \bar{y}^{n+1}\right\} \subseteq \Min(F(\bar{x}),K),\; \left\{\bar{v}^1,\ldots, \bar{v}^{n+1}\right\} \subseteq K^*, \;\left\{\bar{u}^1,\ldots, \bar{u}^{n+1}\right\} \subseteq \R^n: \\
\text{\textbullet} \;\; 0 \in \sum\limits_{i=1}^{n+1} \bar{u}^i + N(\bar{x},\Omega),\\
\text{\textbullet} \;\;\forall\; i \in [n+1]: (\bar{u}^i,-\bar{v}^i) \in N\left((\bar{x},\bar{y}^i), \gph F \right),\\
\text{\textbullet} \;\; (\bar{v}^1,\ldots,\bar{v}^{n+1}) \neq 0,
\end{cases}
\end{equation*} which can be rewritten as 
\begin{equation*}
    \begin{cases}
\exists \left\{\bar{y}^1,\ldots \bar{y}^{n+1}\right\} \subseteq \Min(F(\bar{x}),K),\; \left\{\bar{v}^1,\ldots, \bar{v}^{n+1}\right\} \subseteq K^*: \\
\text{\textbullet} \;\; 0 \in \begin{pmatrix} 0 \\ \bar{v}^1\\ \vdots \\\bar{v}^{n+1}\end{pmatrix} + \sum\limits_{i=1}^{n+1} \left\{ \left(u,0,\ldots,\smash[b]{\smallunderbrace{v}_{\mathclap{(i + 1)^{th} \textrm{ term}} }},\ldots,0\right) \mid \left(u,v\right) \in N\left( (\bar{x},\bar{y}^i ),\gph F\right)\right\} + N(\bar{x},\Omega) \times \{0\},\\
\text{\textbullet} \;\; (\bar{v}^1,\ldots,\bar{v}^{n+1}) \neq 0.
\end{cases}
\end{equation*} Next, according to Lemma \ref{lem:ncintGi}, it follows that 
\begin{equation*}
    \begin{cases}
\exists \left\{\bar{y}^1,\ldots \bar{y}^{n+1}\right\} \subseteq \Min(F(\bar{x}),K),\; \left\{\bar{v}^1,\ldots, \bar{v}^{n+1}\right\} \subseteq K^*: \\
\text{\textbullet} \;\; 0 \in \begin{pmatrix} 0 \\ \bar{v}^1\\ \vdots \\\bar{v}^{n+1} \end{pmatrix} + N\left( (\bar{x},\bar{y}^1,\ldots,\bar{y}^{n + 1} ),\gph F^{n + 1}\right) + N(\bar{x},\Omega) \times \{0\},\\
\text{\textbullet} \;\; (\bar{v}^1,\ldots,\bar{v}^{n+1}) \neq 0.
\end{cases}
\end{equation*}
Applying now Lemma \ref{lem:ncintgphFomega}, the above is equivalent to
\begin{equation}\label{eq:preKKT}
    \begin{cases}
\exists \left\{\bar{y}^1,\ldots \bar{y}^{n+1}\right\} \subseteq \Min(F(\bar{x}),K),\; \left\{\bar{v}^1,\ldots, \bar{v}^{n+1}\right\} \subseteq K^*: \\
\text{\textbullet} \;\; 0 \in \begin{pmatrix} 0 \\ \bar{v}^1\\ \vdots \\\bar{v}^{n+1}\end{pmatrix} + N\left( (\bar{x},\bar{y}^1,\ldots,\bar{y}^{n + 1} ),\gph F^{n + 1} \cap \left(\Omega \times \prod\limits_{i=1}^{n + 1}\R^m\right)\right),\\
\text{\textbullet} \;\; (\bar{v}^1,\ldots,\bar{v}^{n+1}) \neq 0.
\end{cases}
\end{equation} Finally, taking into account that
\begin{equation*}
\nabla f^{n+1}(\bar{x},\bar{y}^1,\ldots,\bar{y}^{n+1}) = \begin{pmatrix}
0_{n\times q}  \\
\mathcal{I}_{q} 
\end{pmatrix}, \; q:= m(n+1),
\end{equation*} we deduce from \eqref{eq:preKKT} the existence of  $\left\{\bar{y}^1,\ldots \bar{y}^{n+1}\right\} \subseteq \Min(F(\bar{x}),K)$ and a vector $\bar{w}:= \left(\bar{v}^1,\ldots, \bar{v}^{n+1}\right) \in  \left(K^{n + 1}\right)^* \setminus \{0\}$ such that 
$$0 \in \nabla f^{n+1}(\bar{x},\bar{y}^1,\ldots,\bar{y}^{n+1}) \bar{w} + N\left( (\bar{x},\bar{y}^1,\ldots,\bar{y}^{n + 1} ),\gph F^{n + 1} \cap \left(\Omega \times \prod\limits_{i=1}^{n + 1}\R^m\right)\right).$$
Then, in particular, we find that the point $(\bar{x}, \bar{y}^1,\ldots, \bar{y}^{n + 1})$ satisfies a standard first order necessary and sufficient optimality condition (see \cite[Theorem 4.14]{MordukhovichNam2014path}) for the convex optimization problem 
\begin{equation}\label{eq:cvp}
\begin{array}{ll}
\min\limits_{x, y^1,\ldots,y^{n + 1}} &\;  \bar{w}^\top f^{n+1}(x, y^1,\ldots, y^{n + 1})  \\
\;\;\;\;\;\textup{s.t.} & (x, y^1,\ldots, y^{n + 1}) \in \gph F^{n + 1} \cap \left(\Omega \times \prod\limits_{i=1}^{n + 1}\R^m\right).\\
\end{array}  
\end{equation} Thus, $(\bar{x}, \bar{y}^1,\ldots, \bar{y}^{n + 1})$ is an optimal solution of \eqref{eq:cvp}. This, together with the fact that  \eqref{eq:cvp} is a weighted sum scalarized problem for $(\mathcal{VP}_{n+1})$ and that $\bar{w} \in \left(K^{n + 1}\right)^* \setminus \{0\},$ allow us to apply \cite[Theorem 5.28]{Jahn2011} to obtain that $(\bar{x},\bar{y}^1,\ldots,\bar{y}^{n+1}) \in \wargmin (\mathcal{VP}_{n+1}).$  The proof is complete.
\end{proof}

\section{Conclusions}
\label{section:Conclusions}
This paper introduces, for the first time, a practical methodology for solving general nonconvex set optimization problems with respect to the set approach. The parametric family of multiobjective subproblems presented is able to completely describe or at least approximate the set optimization problem with arbitrary precision while maintaining useful properties such as convexity. Consequently, our results significantly reduce the complexity of many classes of set optimization problems that usually arise both in the literature and in applications. Finally, many tools and techniques from set optimization with respect to the vector approach can now be applied in this setting. Thus, further research in this direction is of interest.

\section*{Funding}
The research of the second author is funded by the  German Federal Ministry for Economic Affairs and Energy (BMWi) under grant 03EI4013B. The research of the third author is funded by  DFG under grant no. 392195690. 

\bibliographystyle{siam}
\bibliography{references}
\end{document}